\documentclass[11pt,reqno]{amsart}

\usepackage{amssymb}
\usepackage[margin=1in]{geometry}
\usepackage[hidelinks,bookmarksnumbered]{hyperref}
\usepackage{expl3}

\title{Even-degree Hermitian Ikeda lifts via Fourier--Jacobi descent}
\author[N. Takeda]{Nobuki TAKEDA}
\address{Department of Mathematics, Graduate School of Science, Kyoto University, Kyoto 606-8502, Japan}
\email{takeda.nobuki.z04@kyoto-u.jp}

\makeatletter
\let\ams@tocwrite\@tocwrite
\def\@tocwrite#1#2{%
  \ifx\@secnumber\@empty
  \else
    \ams@tocwrite{#1}{#2}%
  \fi
}
\makeatother

\theoremstyle{definition}
\newtheorem{dfn}{Definition}[section]
\newtheorem{rem}[dfn]{Remark}

\theoremstyle{plain}
\newtheorem{thm}[dfn]{Theorem}
\newtheorem{prop}[dfn]{Proposition}
\newtheorem{lem}[dfn]{Lemma}
\newtheorem{cor}[dfn]{Corollary}

\DeclareMathOperator{\tr}{Tr}
\DeclareMathOperator{\Hom}{Hom}
\newcommand{\Std}{\mathrm{Std}}

\ExplSyntaxOn

\cs_new:Npn \makealphabet #1 #2 #3
  {
    \clist_map_inline:nn {#3}
      { \exp_args:Nc \newcommand {#1 ##1} {#2{##1}} }
  }

\makealphabet{bb}{\mathbb}{A,C,Q,R,Z}
\makealphabet{bf}{\mathbf}{a,h,i}
\makealphabet{cal}{\mathcal}{A,D,F,G,I,J,L,M,O,S,T,V,W}
\makealphabet{frak}{\mathfrak}{H,O,S}

\clist_map_inline:nn
  {
    GU,U,GL,Her,Res,SL,Ind,FJ,hol,Lift,
    Asai,sph,Gal,Sym,vol,disc,BC,diag,
    st,Wh,id,irr,adm,ord,pr,spl,AI,std
  }
  {
    \exp_args:Nc \newcommand {#1} {\mathrm{#1}}
  }

\ExplSyntaxOff

\newcommand{\one}{\mathbf{1}}

\renewcommand{\sc}{\mathrm{sc}}
\renewcommand{\dh}{\mathrm{dh}}
\renewcommand{\Im}{\mathrm{Im}}
\renewcommand{\Re}{\mathrm{Re}}

\numberwithin{equation}{section}
\allowdisplaybreaks
\begin{document}
\begin{abstract}
  Let \(E/F\) be a CM extension and let \(m\geq2\) be an even integer.
  Starting from the odd-degree Hermitian Ikeda lift of degree \(m+1\) constructed by Yamana,
  we construct a Hermitian Ikeda lift of degree \(m\) from the theta components of its first Fourier--Jacobi coefficient.
  We determine its global Arthur parameter and its local constituents at every finite place.
  Using the global multiplicity formula for unitary groups,
  we obtain an explicit multiplicity-free decomposition of the resulting representation of \(\U_{m,m}(\bbA_{F,\bfh})\)
  and show that its isomorphism class is independent of the Fourier--Jacobi index.
  When \(F=\bbQ\),
  we compare Fourier coefficients and identify the construction with Ikeda's even-degree Hermitian lift.
\end{abstract}

\maketitle

\tableofcontents

\section{Introduction}

For positive integers \(k\) and \(n\) with \(k\equiv n\pmod 2\),
Ikeda constructed a lift from normalized Hecke eigenforms in \(S_{2k}(\SL_2(\bbZ))\) to Siegel cusp forms of degree \(2n\) and weight \(k+n\),
proving the conjecture of Duke and Imamo\u{g}lu \cite{DukeImamoglu1996Converse,Ikeda2001Siegel}.
He subsequently established a Hermitian analogue over an imaginary quadratic field in both even and odd degrees,
giving explicit formulas for its Fourier coefficients as well as an adelic formulation \cite{Ikeda2008Lifting}.

These lifts have been studied extensively from an arithmetic point of view.
Their periods and related formulas for special values have been investigated in \cite{KatsuradaKawamura2015Period,Katsurada2017HermitianPeriod,Higashitani2026Periods}.
They also play an important role in congruence problems,
in particular those related to Harder's conjecture;
see \cite{Dummigan2017Congruences,AtobeChidaIbukiyamaKatsuradaYamauchi2023Harder}.

The adelic theory was later extended in two directions.
Ikeda and Yamana constructed Hilbert--Siegel lifts over totally real fields \cite{IkedaYamana2020HilbertSiegel},
while Yamana developed a Hermitian lifting over a general CM extension \(E/F\) \cite{Yamana2020HilbertHermitian}.
Yamana's construction produces cuspidal automorphic forms on the quasi-split unitary similitude groups \(\GU_{r,r}\) in odd degree.
As observed in \cite[Introduction, after~(1.3)]{Yamana2020HilbertHermitian},
the corresponding even-degree forms on \(\U_{m,m}\) are expected to arise from the first Fourier--Jacobi coefficient of the odd-degree lift.
The purpose of this paper is to carry out this construction and to determine the automorphic representations generated by the resulting forms.

Let \(m\geq2\) be even.
Starting from the Hermitian Ikeda lift of degree \(m+1\) constructed by Yamana,
we take its first Fourier--Jacobi coefficient and then its theta components.
This produces holomorphic automorphic forms on \(\U_{m,m}\) and gives the natural candidate for the even-degree lift.

At a nonsplit finite place,
the restriction of the local representation of the odd-degree lift from the unitary similitude group to the unitary group may be reducible.
One must therefore determine the Fourier--Jacobi descent of each odd-degree constituent.
We carry this out at every finite place,
including the ramified nonsupercuspidal and dihedral supercuspidal cases.

We prove that the resulting automorphic forms are nonzero and cuspidal,
identify their global Arthur packet using Mok's endoscopic classification,
and determine all finite local constituents.
Combining the local analysis with the global multiplicity formula gives an explicit multiplicity-free decomposition of \(\Pi_{m,S}(h)\).
In particular,
although the construction depends a priori on a positive Fourier--Jacobi index \(S\),
the isomorphism class of the resulting representation of \(G_m(\bbA_\bfh)\) is independent of \(S\).

When \(F=\bbQ\),
the construction recovers the even-degree Hermitian lift in~\cite{Ikeda2008Lifting}.
In this specialization,
we determine the local constituent contributing at full level,
including at primes dividing the discriminant,
compare the Fourier coefficients with the formula in~\cite{Ikeda2008Lifting},
and recover the exceptional vanishing phenomenon in \cite[Corollary~15.21]{Ikeda2008Lifting} from the resulting global compatibility condition.
In particular,
our construction gives a representation-theoretic realization of the Arthur parameter anticipated in \cite[\S~18]{Ikeda2008Lifting}.

We now state the main result.
Let \(h\) be a normalized primitive Hilbert Hecke eigenform of weight \(\kappa\),
and let \(\pi_h\) be the associated automorphic representation of \(\GL_2(\bbA_F)\).
We use the auxiliary characters and weights fixed in Subsection~\ref{sec:automorphic-forms},
and the notation for Arthur parameters and standard \(L\)-functions introduced in Section~\ref{sec:arthur-packets}.
If \(h\) has CM by \(E\) in the sense of Subsection~\ref{subsec:arthur-parameters-lifts},
write \(\pi_h\simeq\AI_{E/F}(\theta)\),
where \(\theta\) is a Hecke character of \(E\),
unique up to conjugation.
For \(S\in F_{>0}\),
Section~\ref{sec:construction} constructs a representation \(\Pi_{m,S}(h)\) of \(G_m(\bbA_\bfh)\) and a \(G_m(\bbA_\bfh)\)-equivariant map \(I_{m,S}^{\kappa}(h;\,\cdot)\) into \(\calS_m^{\ell_m,\eta_m}\).

\begin{thm}
  \label{thm:introduction-main}
  Let \(m\geq2\) be even,
  and suppose that \(\eta_{m+1}\in\bbZ^{\bfa}\).
  Then,
  for every \(S\in F_{>0}\),
  the following assertions hold.
  \begin{enumerate}
    \item The representation \(\Pi_{m,S}(h)\) is nonzero,
          and
          \[
            I_{m,S}^{\kappa}(h;\,\cdot):\Pi_{m,S}(h)\hookrightarrow\calS_m^{\ell_m,\eta_m}
          \]
          is a nonzero \(G_m(\bbA_\bfh)\)-equivariant embedding.

    \item Every irreducible automorphic constituent \(\sigma\) occurring in the image of \(I_{m,S}^{\kappa}(h;\,\cdot)\) belongs to the global Arthur packet attached to
          \[
            \begin{cases}
              (((\vartheta_m)^c)^{-1}\otimes\BC_{E/F}(\pi_h))[m]                           & \text{if }\pi_h\not\simeq\pi_h\otimes\chi_{E/F}, \\
              (((\vartheta_m)^c)^{-1}\theta)[m]\boxplus(((\vartheta_m)^c)^{-1}\theta^c)[m] & \text{if }\pi_h\simeq\AI_{E/F}(\theta),
            \end{cases}
          \]
          and
          \[
            L(s,\sigma,\Std)=\prod_{j=0}^{m-1}L\left(s+\frac{m-1}{2}-j,((\vartheta_m)^c)^{-1}\otimes\BC_{E/F}(\pi_h)\right).
          \]

    \item Every such automorphic representation occurs with multiplicity one,
          and the isomorphism class of \(\Pi_{m,S}(h)\) is independent of \(S\).
  \end{enumerate}
\end{thm}

More precisely,
at every finite place where the relevant local packet has two members,
Fourier--Jacobi descent sends the constituent indexed by \(\delta_v\) to that indexed by \(\delta_v\chi_v(S_v)\):
via the local intertwining map constructed by Yamana at points where the principal series is reducible,
and via an isomorphism of the full Fourier--Jacobi module at places \(v\) where \(\pi_{h,v}\) is dihedral supercuspidal.
Together with the global multiplicity formula established by Mok,
this gives the independence of \(S\) and,
over \(\bbQ\),
the exceptional vanishing phenomenon.

The paper is organized as follows.
Section~\ref{sec:preliminaries} fixes the notation for unitary groups and automorphic forms,
introduces the auxiliary characters and weights used in the construction,
and recalls Fourier--Jacobi coefficients and their theta decomposition.
Section~\ref{sec:construction} recalls Yamana's odd-degree lift,
constructs the even-degree lift by Fourier--Jacobi descent,
and proves its nonvanishing.
Section~\ref{sec:arthur-packets} recalls the endoscopic classification,
defines the Arthur parameters of the lifts,
computes their unramified local parameters,
and identifies the global packets containing these lifts.
Section~\ref{sec:local-fj-descent} determines the local constituents and their Fourier--Jacobi descent at every finite place.
Section~\ref{sec:finite-part} combines the local results with the global multiplicity formula to obtain the global decomposition of \(\Pi_{m,S}(h)\) and prove its independence of the Fourier--Jacobi index.
Finally,
Section~\ref{sec:Q-comparison} specializes to \(F=\bbQ\) and compares the construction with Ikeda's even-degree Hermitian lift.

\subsection*{Acknowledgements}
The author is grateful to T. Ikeda for his guidance and support.
This work was supported by JSPS KAKENHI Grant Number JP26KJ1378.
\subsection*{Notation}

Throughout the paper,
\(F\) is a totally real field,
\(E/F\) is a CM extension,
and \(c\) is the nontrivial automorphism of \(E/F\).
Write \(\bfa\) and \(\bfh\) for the archimedean and finite places of \(F\).
Write \(\bbA_F\) and \(\bbA_E\) for the adele rings,
and \(\bbA_{F,\bfh}\) and \(\bbA_{E,\bfh}\) for their finite parts.

For a local or adelic vector space \(X\),
write \(\calS(X)\) for its Schwartz--Bruhat space or the corresponding restricted tensor product.
Let \(\chi_{E/F}\) be the quadratic Hecke character associated with \(E/F\).
The symbol \(\one\) denotes the trivial character of the group under consideration.
For a set \(X\),
the symbol \(1_X\) denotes its characteristic function.
We identify the signs \(\pm\) with \(\pm1\).

Write \(F_{>0}\) for the set of totally positive elements of \(F\).
For a character \(\lambda\) of \(E^\times\),
\(E_v^\times\),
or the corresponding idele group,
write \(\lambda^c(x)=\lambda(x^c)\) and \(\lambda^{-c}=(\lambda^c)^{-1}\).
Write \(N_{E/F}\) and \(\tr_{E/F}\) for the norm and trace from \(E\) to \(F\).

For a place \(v\) of \(F\),
let \(F_v\) be the completion of \(F\) at \(v\),
and put \(E_v=E\otimes_F F_v\) and \(\chi_v=\chi_{E_v/F_v}\).
Write \(N_{E_v/F_v}\) and \(\tr_{E_v/F_v}\) for the corresponding local norm and trace.
We also put \(F_\infty=\prod_{v\in\bfa}F_v\) and \(E_\infty=\prod_{v\in\bfa}E_v\).
For \(v\in\bfh\),
write \(\calO_{F_v}\) for the ring of integers of \(F_v\),
and \(\calO_{E_v}\) for its integral closure in \(E_v\).
We use the standard normalized absolute values on local fields.
For \(x=(x_v)_{v\in\bfa}\in F_\infty\),
put \(|x|_\infty=\prod_{v\in\bfa}|x_v|_v\).
For square matrices \(A\) and \(B\),
write \(A\oplus B=\diag(A,B)\) for their block diagonal sum.

For \(x=(x_v)_{v\in\bfa}\in\bbC^{\bfa}\),
put
\[
  e_\infty(x)
  =
  \prod_{v\in\bfa}\exp(2\pi\sqrt{-1}\,x_v).
\]

For an \(F\)-algebra \(R\) and an integer \(a\geq1\),
put
\[
  \Her_a(R)
  =
  \{S\in M_a(E\otimes_F R)\mid S^*=S\},
\]
where \(S^*={}^tS^c\) denotes the conjugate transpose of \(S\).
For \(A\in\GL_a(E\otimes_F R)\) and \(B\in\Her_a(R)\),
put \(B[A]=A^*BA\).
For a place \(v\) of \(F\),
write \(S_v\) for the image of \(S\in\Her_a(F)\) in \(\Her_a(F_v)\),
and put
\[
  \Her_a(F)_{>0}
  =
  \{S\in\Her_a(F)\mid S_v>0\text{ for every }v\in\bfa\}.
\]

\section{Preliminaries}
\label{sec:preliminaries}

\subsection{Unitary groups and parabolic subgroups}

For \(r\geq1\),
put
\[
  J_r=
  \begin{pmatrix}
    0    & 1_r \\
    -1_r & 0
  \end{pmatrix}.
\]
For every \(F\)-algebra \(R\),
define the unitary similitude group \(\widetilde G_r=\GU_{r,r}\) by
\[
  \widetilde G_r(R)
  =
  \left\{
  g\in\GL_{2r}(E\otimes_F R)
  \ \middle|\
  g^*J_rg=\nu_r(g)J_r
  \text{ for some }\nu_r(g)\in R^\times
  \right\}.
\]
The scalar \(\nu_r(g)\) defines the similitude character \(\nu_r:\widetilde G_r\rightarrow\mathbb G_m\).
We define the unitary group \(G_r=\U_{r,r}\) by \(G_r=\ker(\nu_r)\),
so that
\[
  G_r(R)
  =
  \left\{
  g\in\GL_{2r}(E\otimes_F R)
  \ \middle|\
  g^*J_rg=J_r
  \right\}.
\]
We write simply \(\nu\) for \(\nu_r\) when the degree is clear.

For \(A\in\GL_r(E\otimes_F R)\),
\(Z\in\Her_r(R)\),
and \(\xi\in R^\times\),
put
\[
  \begin{aligned}
    m_r(A)   & =\begin{pmatrix}A&0\\0&(A^*)^{-1}\end{pmatrix}, \\
    n_r(Z)   & =\begin{pmatrix}1_r&Z\\0&1_r\end{pmatrix},      \\
    d_r(\xi) & =\begin{pmatrix}1_r&0\\0&\xi1_r\end{pmatrix}.
  \end{aligned}
\]

Write \(P_r\) for the Siegel parabolic subgroup of \(\widetilde G_r\),
whose \(R\)-points are
\[ P_r(R)
  =
  \left\{
  \begin{pmatrix}
    A & B             \\
    0 & \xi(A^*)^{-1}
  \end{pmatrix}
  \,\middle|\,
  A\in\GL_r(E\otimes_F R),\ \xi\in R^\times,\ A^{-1}B\in\Her_r(R)
  \right\}. \]
Its unipotent radical is \(N_r(R)=\{n_r(Z)\mid Z\in\Her_r(R)\}\).

For \(r\geq2\),
let \(Q_{r,1}=L_{r,1}U_{r,1}\) be the maximal parabolic subgroup stabilizing the first isotropic line in \(E^{2r}\).
With respect to
\[
  E^{2r}=E\oplus E^{r-1}\oplus E\oplus E^{r-1},
\]
its Levi subgroup is
\[
  L_{r,1}\simeq \Res_{E/F}\GL_1\times\widetilde G_{r-1}.
\]
If \(g=\begin{pmatrix}\alpha&\beta\\ \gamma&\delta\end{pmatrix}\in\widetilde G_{r-1}(R)\),
write
\[
  l_{r,1}(a,g)
  =
  \begin{pmatrix}
    a & 0      & 0                & 0      \\
    0 & \alpha & 0                & \beta  \\
    0 & 0      & \nu(g)(a^*)^{-1} & 0      \\
    0 & \gamma & 0                & \delta
  \end{pmatrix}.
\]
The unipotent radical consists of
\[
  u_{r,1}(X,Y,t)
  =
  \begin{pmatrix}
    1 & X       & t+XY^* & Y       \\
    0 & 1_{r-1} & Y^*    & 0       \\
    0 & 0       & 1      & 0       \\
    0 & 0       & -X^*   & 1_{r-1}
  \end{pmatrix},
\]
where \(X,Y\in M_{1,r-1}(E\otimes_F R)\) and \(t\in R\).
Its center is given on every \(F\)-algebra \(R\) by
\[
  Z_{r,1}(R)
  =
  \{z_{r,1}(t)=u_{r,1}(0,0,t):t\in R\}.
\]
We use the standard embedding \(\iota_{r,1}:G_{r-1}\hookrightarrow G_r\),
given by \(g\mapsto l_{r,1}(1,g)\).

\subsection{Hermitian modular and automorphic forms}
\label{sec:automorphic-forms}

We fix the classical and adelic realizations of Hermitian automorphic forms used below.

\subsubsection{Classical Hermitian modular forms}

For \(v\in\bfa\),
let
\[
  \frakH_{n,v}
  =
  \left\{
  Z\in M_n(\bbC)
  \ \middle|\
  \frac{Z-Z^*}{2\sqrt{-1}}>0
  \right\},
\]
and put
\[
  \frakH_n^{\bfa}
  =
  \prod_{v\in\bfa}\frakH_{n,v}.
\]
We write
\[
  \bfi_n
  =
  (\sqrt{-1}\,1_n)_{v\in\bfa}
\]
for the standard base point.

Let \(\GU^+(n,n)=\{g\in\GU(n,n)(\bbR):\nu(g)>0\}\),
and put \(\widetilde G_{n,\infty}^+=\prod_{v\in\bfa}\GU^+(n,n)\).
If \(g=\left(\begin{matrix} A_v & B_v \\ C_v & D_v
\end{matrix}\right)_{v\in\bfa}\),
then \(g\) acts on \(\frakH_n^{\bfa}\) by
\[
  g\langle Z\rangle
  =
  \left(
  (A_vZ_v+B_v)(C_vZ_v+D_v)^{-1}
  \right)_{v\in\bfa}.
\]

For \(\ell=(\ell_v)_{v\in\bfa}\in\bbZ^{\bfa}\) and \(g\in\widetilde G_{n,\infty}^+\),
define the automorphy factor
\[
  j_\ell(g,Z)
  =
  \prod_{v\in\bfa}
  \det(C_vZ_v+D_v)^{\ell_v}.
\]
For \(\eta=(\eta_v)_{v\in\bfa}\in\bbZ^{\bfa}\),
define the character \(\varepsilon^\eta:E_\infty^\times\rightarrow\bbC^\times\) by
\[
  \varepsilon^\eta(z)
  =
  \prod_{v\in\bfa}
  \left(\frac{z_v}{|z_v|}\right)^{\eta_v}.
\]
For a function \(F:\frakH_n^{\bfa}\rightarrow\bbC\) and \(g\in\widetilde G_{n,\infty}^+\),
define the slash action by
\[
  (F\vert_{\ell,\eta}g)(Z)
  =
  \varepsilon^\eta(\det g)\,
  j_\ell(g,Z)^{-1}
  F(g\langle Z\rangle).
\]

Let \(\widetilde G_n^+(F)\) be the subgroup of \(\widetilde G_n(F)\) consisting of elements with totally positive similitude factor.

\begin{dfn}
  Let \(\Gamma\subset\widetilde G_n^+(F)\) be an arithmetic subgroup.
  A classical Hermitian modular form of weight \((\ell,\eta)\) for \(\Gamma\) is a holomorphic function \(F:\frakH_n^{\bfa}\rightarrow\bbC\) satisfying
  \[
    F\vert_{\ell,\eta}\gamma=F
  \]
  for all \(\gamma\in\Gamma\).
  We denote the space of such forms by \(M_{\ell,\eta}(\Gamma)\).
  A form is called cuspidal if its Fourier expansion at every cusp is supported on positive-definite Hermitian matrices.
  The corresponding subspace is denoted by \(S_{\ell,\eta}(\Gamma)\).
\end{dfn}

The same notation will be used for arithmetic subgroups of \(G_n(F)\),
with the action obtained by restriction from \(\widetilde G_n^+(F)\).

\subsubsection{Adelic Hermitian automorphic forms}

Write \(\widetilde G_n^+(\bbA)\) for the subgroup of \(\widetilde G_n(\bbA)\) with positive archimedean similitude factors.
Then \(\widetilde G_n^+(F)=\widetilde G_n(F)\cap\widetilde G_n^+(\bbA)\),
and
\[
  \widetilde G_n^+(F)\backslash\widetilde G_n^+(\bbA)
  \simeq
  \widetilde G_n(F)\backslash\widetilde G_n(\bbA).
\]
We use this identification without further comment.
Fix open compact subgroups \(\widetilde K_\bfh\subset\widetilde G_n(\bbA_\bfh)\) and \(K_\bfh\subset G_n(\bbA_\bfh)\).
For the rest of this subsection,
use the notation
\[
  (\calG(\bbA),\calG(F),\calG_\infty,K_{\calG,\bfh})
  =
  \begin{cases}
    (\widetilde G_n^+(\bbA),\widetilde G_n^+(F),\widetilde G_{n,\infty}^+,\widetilde K_\bfh)
     & (\text{in the similitude case}), \\
    (G_n(\bbA),G_n(F),G_n(F_\infty),K_\bfh)
     & (\text{in the unitary case}).
  \end{cases}
\]
Put \(K_{\calG,\infty}=\{k\in\calG_\infty\mid k\langle\bfi_n\rangle=\bfi_n\}\).

\begin{dfn}
  A Hermitian automorphic form on \(\calG(\bbA)\) of weight \((\ell,\eta)\) and level \(K_{\calG,\bfh}\) is a smooth function \(f:\calG(F)\backslash\calG(\bbA)\rightarrow\bbC\) satisfying
  \[
    f(gk_\bfh k_\infty)
    =\varepsilon^\eta(\det k_\infty)j_\ell(k_\infty,\bfi_n)^{-1}f(g)
  \]
  for \(k_\bfh\in K_{\calG,\bfh}\) and \(k_\infty\in K_{\calG,\infty}\),
  together with the following holomorphy condition.
  For \(x\in\calG(\bbA_\bfh)\),
  choose \(g_\infty\in\calG_\infty\) such that \(Z=g_\infty\langle\bfi_n\rangle\),
  and define the corresponding classical component by
  \[
    f_x(Z)
    =\varepsilon^\eta(\det g_\infty)^{-1}
    j_\ell(g_\infty,\bfi_n)f(xg_\infty).
  \]
  This is independent of the choice of \(g_\infty\).
  Put \(\Gamma_x=\{\gamma\in\calG(F)\mid\gamma_\bfh\in xK_{\calG,\bfh}x^{-1}\}\).
  We require \(f_x\in M_{\ell,\eta}(\Gamma_x)\) for every \(x\in\calG(\bbA_\bfh)\).
  We denote the resulting space by \(\calA_{\calG}^{\ell,\eta}(K_{\calG,\bfh})\),
  which means \(\widetilde{\calA}_n^{\ell,\eta}(\widetilde K_\bfh)\) in the similitude case and \(\calA_n^{\ell,\eta}(K_\bfh)\) in the unitary case.
\end{dfn}

\begin{dfn}
  Let \(f\in\calA_{\calG}^{\ell,\eta}(K_{\calG,\bfh})\).
  For every proper \(F\)-parabolic subgroup \(P\) of the corresponding algebraic group,
  write \(N_P\) for its unipotent radical.
  The form \(f\) is cuspidal if
  \[
    \int_{N_P(F)\backslash N_P(\bbA)}f(ug)\,du=0
  \]
  for every such \(P\) and every \(g\) in the corresponding adelic group.
  We denote the resulting subspace by \(\calS_{\calG}^{\ell,\eta}(K_{\calG,\bfh})\),
  which means \(\widetilde{\calS}_n^{\ell,\eta}(\widetilde K_\bfh)\) in the similitude case and \(\calS_n^{\ell,\eta}(K_\bfh)\) in the unitary case.
\end{dfn}

Choose representatives \(x_1,\ldots,x_h\in\calG(\bbA_\bfh)\) for the decomposition
\[
  \calG(\bbA)
  =\coprod_{i=1}^h
  \calG(F)x_iK_{\calG,\bfh}\calG_\infty,
\]
and put
\[
  \Gamma_{x_i}
  =
  \{\gamma\in\calG(F)\mid\gamma_\bfh\in x_iK_{\calG,\bfh}x_i^{-1}\}.
\]
The map \(f\mapsto(f_{x_i})_{i=1}^h\) induces the classical--adelic isomorphisms
\[
  \begin{aligned}
    \calA_{\calG}^{\ell,\eta}(K_{\calG,\bfh})
     & \xrightarrow{\sim}
    \bigoplus_{i=1}^h M_{\ell,\eta}(\Gamma_{x_i}), \\
    \calS_{\calG}^{\ell,\eta}(K_{\calG,\bfh})
     & \xrightarrow{\sim}
    \bigoplus_{i=1}^h S_{\ell,\eta}(\Gamma_{x_i}).
  \end{aligned}
\]
The inverse sends \((F_i)_{i=1}^h\) to the unique adelic form satisfying
\[
  f(\gamma x_i k_\bfh g_\infty)
  =
  (F_i\vert_{\ell,\eta}g_\infty)(\bfi_n)
\]
for \(\gamma\in\calG(F)\),
\(k_\bfh\in K_{\calG,\bfh}\),
and \(g_\infty\in\calG_\infty\).

Put
\[
  \begin{aligned}
    \widetilde{\calS}_n^{\ell,\eta}
     & =\bigcup_{\widetilde K_\bfh}\widetilde{\calS}_n^{\ell,\eta}(\widetilde K_\bfh), \\
    \calS_n^{\ell,\eta}
     & =\bigcup_{K_\bfh}\calS_n^{\ell,\eta}(K_\bfh).
  \end{aligned}
\]
These spaces carry natural right actions of \(\widetilde G_n(\bbA_\bfh)\) and \(G_n(\bbA_\bfh)\),
respectively.
Restriction from \(\widetilde G_n^+(\bbA)\) to \(G_n(\bbA)\) induces
\[
  \widetilde{\calS}_n^{\ell,\eta}(\widetilde K_\bfh)
  \rightarrow
  \calS_n^{\ell,\eta}\left(\widetilde K_\bfh\cap G_n(\bbA_\bfh)\right).
\]

For \(f\in\calS_{\calG}^{\ell,\eta}(K_{\calG,\bfh})\) and \(x\in\calG(\bbA_\bfh)\),
the classical component \(f_x\) has the Fourier expansion
\[
  f_x(Z)
  =\sum_{B\in\Her_n(F)_{>0}}
  a(f_x,B)e_\infty(\tr(BZ)).
\]
This defines the Fourier coefficients \(a(f_x,B)\) used below.

We now fix the global data used in the construction.
Let \(h\) be a normalized primitive Hilbert Hecke eigenform of weight \(\kappa=(\kappa_v)_{v\in\bfa}\in\bbZ^{\bfa}\),
let \(\pi_h=\bigotimes_v'\pi_{h,v}\) be the automorphic representation of \(\GL_2(\bbA)\) generated by \(h\),
and write \(\omega_h=\prod_v\omega_{h,v}\) for its central character.
We regard an integer as the corresponding parallel weight vector whenever it is added to an element of \(\bbZ^{\bfa}\).
Choose a unitary Hecke character \(\chi_-:E^\times\backslash\bbA_E^\times\rightarrow\bbC^\times\) whose restriction to \(\bbA^\times\) is \(\chi_{E/F}\).
Choose also a unitary Hecke character \(\vartheta^+:E^\times\backslash\bbA_E^\times\rightarrow\bbC^\times\) whose restriction to \(\bbA^\times\) is \(\omega_h\),
and put \(\vartheta^-=\vartheta^+\chi_-^{-1}\).
For every place \(v\),
write \(\chi_{-,v}\) and \(\vartheta_v^\pm\) for the local components of \(\chi_-\) and \(\vartheta^\pm\),
respectively.

For a Hecke character \(\lambda\) of \(E\),
write \(\ell_\infty(\lambda)=(k_v)_{v\mid\infty}\) if
\[
  \lambda_\infty((z_v)_{v\mid\infty})=\prod_{v\mid\infty}\left(\frac{z_v}{|z_v|}\right)^{k_v}.
\]
For \(n\geq1\),
put
\begin{align*}
  \vartheta_n & =\begin{cases}\vartheta^+ & (n\text{ odd}),\\ \vartheta^- & (n\text{ even}),\end{cases} \\
  \ell_n      & =\kappa+n-1,                                                                            \\
  \eta_n      & =\frac{\ell_n+\ell_\infty(\vartheta_n)}{2}.
\end{align*}
For every place \(v\),
write \(\vartheta_{n,v}\) for the localization of \(\vartheta_n\).

\subsection{Fourier--Jacobi coefficients and theta decomposition}

For the adelic theta decomposition underlying Fourier--Jacobi coefficients,
see \cite{Ikeda1994Jacobi}.
The corresponding theta decomposition for Fourier--Jacobi expansions of Hermitian modular forms is described in \cite[\S~6]{Ikeda2008Lifting}.
We follow the normalization of the Weil representation and theta kernel in \cite[\S~2]{Yamana2020HilbertHermitian};
the theta components used below are those of \cite[Proposition~10.4]{Yamana2020HilbertHermitian}.

Let \(\psi_{\bbQ}\) be the standard additive character of \(\bbQ\backslash\bbA_{\bbQ}\) whose archimedean component is \(x\mapsto\exp(2\pi\sqrt{-1}\,x)\),
and put \(\psi=\psi_{\bbQ}\circ\tr_{F/\bbQ}\).
Write \(\psi_v\) for its local component.

We only use the first Fourier--Jacobi coefficient.
Fix \(r\geq2\) and \(S\in F_{>0}\),
and retain the splitting character \(\chi_-\) fixed above.

The unipotent radical \(U_{r,1}\) of \(Q_{r,1}\) is a Heisenberg group over \(F\),
with center
\[
  Z_{r,1}
  =
  \{z_{r,1}(t)=u_{r,1}(0,0,t):t\in\mathbb G_a\}.
\]
The index \(S\) determines the central character
\[
  \psi_S(z_{r,1}(t))=\psi(St)
\]
of \(Z_{r,1}(F)\backslash Z_{r,1}(\bbA)\).
For every place \(v\),
the Stone--von Neumann theorem gives a unique irreducible representation \(\rho_{S_v}\) of \(U_{r,1}(F_v)\),
up to isomorphism,
on which
\[
  \rho_{S_v}(z_{r,1}(t))
  =
  \psi_v(S_vt)\,\mathrm{id}.
\]
We write
\[
  \rho_S={\bigotimes_v}'\rho_{S_v}
\]
for the corresponding adelic Heisenberg representation.
Conjugation through \(\iota_{r,1}:G_{r-1}\hookrightarrow G_r\) preserves \(U_{r,1}\) and fixes its center pointwise.
Using the splitting character \(\chi_-\),
the representations \(\rho_{S_v}\) extend to the local Weil representations \(\omega_{S_v}^{\chi_{-,v}}\) of \(U_{r,1}(F_v)\rtimes G_{r-1}(F_v)\),
with the normalization of \cite[\S~2, especially~(2.5)--(2.6)]{Yamana2020HilbertHermitian}.
Put
\[
  \omega_S^{\chi_-}
  =
  {\bigotimes_v}'\omega_{S_v}^{\chi_{-,v}},
\]
and
\[
  \omega_{S,\bfh}^{\chi_-}
  =
  {\bigotimes_{v\in\bfh}}'\omega_{S_v}^{\chi_{-,v}}.
\]
Thus \(\omega_S^{\chi_-}|_{U_{r,1}(\bbA)}=\rho_S\).
We use the Schr\"odinger realization of \(\rho_S\) on \(L^2(M_{1,r-1}(\bbA_E))\),
and realize \(\omega_S^{\chi_-}\) on the Schwartz--Bruhat space \(\calS(M_{1,r-1}(\bbA_E))\).
We abbreviate \(\omega_S^{\chi_-}\) to \(\omega_S\) when the splitting character is fixed.

We fix the Haar measures used in the Fourier--Jacobi construction as follows.
For every place \(v\),
let \(dt_v\) be the self-dual Haar measure on \(F_v\) with respect to \((s,t)\mapsto\psi_v(st)\).
For \(a\geq1\),
let \(d\mu_{S,v}^{(a)}\) be the self-dual Haar measure on \(M_{1,a}(E_v)\) with respect to \((x,y)\mapsto\psi_v\!\left(\tr_{E_v/F_v}(S_vxy^*)\right)\).
Put
\[
  dt=\prod_vdt_v,
\]
and use the induced quotient measure on \(F\backslash\bbA\).

Via the coordinates \(u_{r,1}(X,Y,t)\),
give \(U_{r,1}(F_v)\) the Haar measure
\[
  du_v
  =
  d\mu_{S,v}^{(r-1)}(X)\,
  d\mu_{S,v}^{(r-1)}(Y)\,
  dt_v.
\]
Put \(du=\prod_vdu_v\),
and use the induced quotient measure on \(U_{r,1}(F)\backslash U_{r,1}(\bbA)\).
For
\[
  Y_{r,1}(F_v)
  =
  \{u_{r,1}(0,Y,0):Y\in M_{1,r-1}(E_v)\},
\]
we use \(d\mu_{S,v}^{(r-1)}(Y)\) and the corresponding quotient measure on \(Y_{r,1}(F_v)\backslash U_{r,1}(F_v)\).

For \(v\in\bfa\) and \(x\in M_{1,r-1}(E_v)\),
put
\[
  \Phi_{S,v}^{\circ}(x)
  =
  \exp(-2\pi S_vxx^*).
\]
For \(\Phi\in\calS(M_{1,r-1}(\bbA_{E,\bfh}))\),
put
\[
  \Phi_S
  =
  \Phi\otimes
  \bigotimes_{v\in\bfa}\Phi_{S,v}^{\circ}.
\]
Following the normalization of the theta kernel in~\cite[\S~2]{Yamana2020HilbertHermitian},
define
\[
  \Theta_S(u,g;\Phi_S)
  =
  \sum_{x\in M_{1,r-1}(E)}
  (\omega_S(u,g)\Phi_S)(x).
\]

\begin{dfn}\label{def:first-FJ-coefficient}
  Let \(f\in\calS_r^{\ell,\eta}(K)\).
  Its first Fourier--Jacobi coefficient of index \(S\) is
  \begin{equation}\label{eq:FJ-coefficient}
    \FJ_S(f)(u,g)
    =
    \int_{F\backslash\bbA}
    f(z_{r,1}(t)u\iota_{r,1}(g))\,
    \overline{\psi(St)}\,dt
  \end{equation}
  for \(u\in U_{r,1}(\bbA)\) and \(g\in G_{r-1}(\bbA)\).
  For \(\Phi\in\calS(M_{1,r-1}(\bbA_{E,\bfh}))\),
  define its theta component,
  in the adelic normalization of~\cite[Proposition~10.4]{Yamana2020HilbertHermitian},
  by
  \[
    \calT_{S,\Phi}(f)(g)
    =
    \int_{U_{r,1}(F)\backslash U_{r,1}(\bbA)}
    f(u\iota_{r,1}(g))\,
    \overline{\Theta_S(u,g;\Phi_S)}\,du.
  \]
\end{dfn}

\begin{rem}
  We use ``first Fourier--Jacobi coefficient'' for the Jacobi form \(\FJ_S(f)\) and ``theta component'' for \(\calT_{S,\Phi}(f)\).
  Yamana calls the theta component a Fourier--Jacobi coefficient.
\end{rem}

\begin{prop}
  \label{prop:fourier-jacobi-level-weight}
  Let \(f\in\calS_r^{\ell,\eta}(K)\).
  Suppose that \(\eta-(\ell_\infty(\chi_-)+1)/2\in\bbZ^{\bfa}\).
  Then
  \[
    \calT_{S,\Phi}(f)
    \in
    \calS_{r-1}^{\,\ell-1,\,
      \eta-(\ell_\infty(\chi_-)+1)/2},
  \]
  where \(\ell-1=(\ell_v-1)_{v\in\bfa}\).
  Moreover,
  for every \(k\in G_{r-1}(\bbA_\bfh)\) satisfying \(\iota_{r,1}(k)\in K\) and \(\omega_{S,\bfh}(k)\Phi=\Phi\),
  we have
  \[
    \calT_{S,\Phi}(f)(gk)=\calT_{S,\Phi}(f)(g)
  \]
  for all \(g\in G_{r-1}(\bbA)\).
\end{prop}

\begin{proof}
  The archimedean Fourier--Jacobi calculation in~\cite[Proposition~10.4]{Yamana2020HilbertHermitian} with \(a=1\) shows that \(\calT_{S,\Phi}(f)\) is holomorphic and has the stated weight.
  Since \(f\) is cuspidal,
  the Fourier expansion of \(\calT_{S,\Phi}(f)\) at each cusp is supported on positive-definite indices.
  Hence \(\calT_{S,\Phi}(f)\) is also cuspidal.
  Let \(k\in G_{r-1}(\bbA_\bfh)\) satisfy \(\iota_{r,1}(k)\in K\) and \(\omega_{S,\bfh}(k)\Phi=\Phi\).
  The right \(K\)-invariance of \(f\) gives
  \[
    f(u\iota_{r,1}(gk))
    =f(u\iota_{r,1}(g)\iota_{r,1}(k))
    =f(u\iota_{r,1}(g)).
  \]
  Since \(\omega_{S,\bfh}(k)\Phi=\Phi\),
  the definition of the theta function gives
  \[
    \Theta_S(u,gk;\Phi_S)
    =\Theta_S(u,g;\omega_S(k)\Phi_S)
    =\Theta_S(u,g;\Phi_S).
  \]
  Substituting these identities into the defining integral for \(\calT_{S,\Phi}(f)\) proves the asserted right invariance.
\end{proof}

\section{Construction of the even-degree Hermitian Ikeda lift}
\label{sec:construction}

\subsection{The odd-degree lift constructed by Yamana}

We recall the odd-degree Hermitian Ikeda lift constructed in~\cite[Theorem~1.1 and \S\S~3--10]{Yamana2020HilbertHermitian}.

Retain the form \(h\),
the representation \(\pi_h\),
and the auxiliary characters fixed in Subsection~\ref{sec:automorphic-forms}.
Write \(\pi_{h,\bfh}=\bigotimes_{v\in\bfh}'\pi_{h,v}\) for the finite part of \(\pi_h\).

For the remainder of this section,
fix an odd integer \(r\) and assume that \(\eta_r\in\bbZ^{\bfa}\).
When a Hecke character is evaluated on finite ideles,
we use the same symbol for its finite part.
Write \(\pi_{h,E,\bfh}\) for the finite part of the base change of \(\pi_h\) to \(\GL_2(\bbA_E)\).

Write \(\widetilde P_r\) for the standard parabolic subgroup of \(\widetilde G_r\) with Levi subgroup \((\Res_{E/F}\GL_2)^{(r-1)/2}\times\widetilde G_1\),
and denote its modulus character on \(\widetilde P_r(\bbA_\bfh)\) by \(\delta_{\widetilde P_r}\).

Following Yamana~\cite[Introduction and \S~4]{Yamana2020HilbertHermitian},
let \(\Pi_r(h)\) denote,
in our notation,
the unique irreducible subrepresentation of
\[
  \Ind_{\widetilde P_r(\bbA_\bfh)}^{\widetilde G_r(\bbA_\bfh)}
  \left(\delta_{\widetilde P_r}^{-1/4}
  \otimes
  \{
  (((\vartheta^+)^c)^{-1}\otimes\pi_{h,E,\bfh})^{\boxtimes(r-1)/2}
  \boxtimes
  ((\vartheta^+)^{-1}\boxtimes\pi_{h,\bfh})
  \}\right).
\]

We use the terminology and normalization in~\cite[Definitions~3.1--3.2]{Yamana2020HilbertHermitian}.
For each finite place \(v\) and nondegenerate \(B_v\in\Her_r(F_v)\),
let
\[
  \calW_{B_v,v}:
  \Pi_{r,v}(h)\longrightarrow\bbC
\]
denote the local Shalika functional attached to \(B_v\) constructed in \cite[\S\S~4 and~7]{Yamana2020HilbertHermitian}.
For \(B\in\Her_r(F)_{>0}\),
with localization \(B_v\) at \(v\),
put
\[
  \calW_B
  =
  \bigotimes_{v\in\bfh}\calW_{B_v,v}:
  \Pi_r(h)\longrightarrow\bbC.
\]
Then
\[
  \calW_B
  (\Pi_r(h)(n_r(z)d_r(\xi)m_r(A))\phi)
  =
  \varepsilon^{-\ell_\infty(\vartheta^+)}(\det A)\,
  \psi(\tr(Bz))\,
  \calW_{\xi^{-1}B[A]}(\phi)
\]
for \(z\in\Her_r(\bbA_\bfh)\),
\(\xi\in F_{>0}\),
\(A\in\GL_r(E)\),
and \(\phi\in\Pi_r(h)\).

For a nonzero vector \(\phi\in \Pi_r(h)\),
set
\[
  \widetilde K_r(\phi)
  =\{k\in\widetilde G_r(\bbA_\bfh)\mid \Pi_r(h)(k)\phi=\phi\},
\]
which is an open compact subgroup of \(\widetilde G_r(\bbA_\bfh)\).
Put
\[
  K_r(\phi)
  =\widetilde K_r(\phi)\cap G_r(\bbA_\bfh).
\]

For \(\widetilde g_\bfh\in\widetilde G_r(\bbA_\bfh)\) and \(Z\in\frakH_r^{\bfa}\),
define
\[
  J_r^\kappa(h;\phi)_{\widetilde g_\bfh}(Z)
  =
  \sum_{B\in\Her_r(F)_{>0}}
  |\det B|_\infty^{\ell_r/2}\,
  \calW_B(\Pi_r(h)(\widetilde g_\bfh)\phi)\,
  e_\infty(\tr(BZ)).
\]

Write \(I_r^\kappa(h;\phi)\) for the restriction of \(J_r^\kappa(h;\phi)\) to \(G_r(\bbA)\).

The following is the lifting theorem in~\cite[Theorem~1.1]{Yamana2020HilbertHermitian},
rewritten in our notation.
Yamana denotes our \(\Pi_r(h)\) by \(\Pi_f\),
and denotes our \(\calW_B\) by \(\gimel_B^{\widehat\chi_f}\) in~\cite[(1.2)]{Yamana2020HilbertHermitian},
where \(\widehat\chi_f\) corresponds to \(\vartheta^+_\bfh\) in our notation.

\begin{thm}
  \label{thm:odd-degree-lift}
  For odd \(r\),
  the series \(J_r^\kappa(h;\phi)_{\widetilde g_\bfh}\) converges for every \(\widetilde g_\bfh\),
  and defines an element of \(\widetilde{\calS}_r^{\ell_r,\eta_r}(\widetilde K_r(\phi))\).
  The map \(J_r^\kappa(h;\,\cdot):\Pi_r(h)\hookrightarrow\widetilde{\calS}_r^{\ell_r,\eta_r}\) is \(\widetilde G_r(\bbA_\bfh)\)-equivariant,
  while \(I_r^\kappa(h;\,\cdot):\Pi_r(h)\rightarrow\calS_r^{\ell_r,\eta_r}\) is \(G_r(\bbA_\bfh)\)-equivariant.
  For \(\phi\in \Pi_r(h)\),
  \(g_\bfh\in G_r(\bbA_\bfh)\),
  and \(B\in\Her_r(F)_{>0}\),
  \[
    a(I_r^\kappa(h;\phi)_{g_\bfh},B)
    =
    |\det B|_\infty^{\ell_r/2}\,
    \calW_B(\Pi_r(h)(g_\bfh)\phi).
  \]
\end{thm}

\begin{rem}
  If \(\phi=\bigotimes_v'\phi_v\in \Pi_r(h)\),
  \(\pi_{h,v}\) and \(\vartheta_v^+\) are unramified,
  and \(\phi_v\) is spherical,
  then \(K_r(\phi)_v=G_r(\calO_{F_v})\).
\end{rem}

Unlike the map \(J_r^\kappa(h;\,\cdot)\),
\(I_r^\kappa(h;\,\cdot)\) need not be injective.

\subsection{Fourier--Jacobi construction}

The discussion in~\cite[Introduction, after~(1.3)]{Yamana2020HilbertHermitian} suggests that the even-degree forms should arise from the first Fourier--Jacobi coefficient of the degree \(m+1\) lift.
Over \(\bbQ\), the corresponding Fourier--Jacobi expansion and theta decomposition are already present in Ikeda's construction; see \cite[\S~6]{Ikeda2008Lifting}.
The even-degree case over a general CM extension is not treated by Yamana,
owing to the more intricate local representation theory in even degree and the possible reducibility of the restriction of \(\Pi_{m+1}(h)\) to the unitary group.
In this section,
we carry out the proposed construction by extracting the theta components.

Fix an even integer \(m\geq2\),
and put \(r=m+1\).
Recall that \(\vartheta^-|_{\bbA^\times}=\omega_h\chi_{E/F}\).
For \(\phi\in \Pi_{m+1}(h)\),
Theorem~\ref{thm:odd-degree-lift} gives
\[
  I_{m+1}^{\kappa}(h;\phi)
  \in
  \calS_{m+1}^{\ell_{m+1},\eta_{m+1}}(K_{m+1}(\phi)),
\]

Since \(m\) is even,
we have \(\ell_m=\kappa+m-1\) and \(\eta_m=(\ell_m+\ell_\infty(\vartheta^-))/2\).
For every archimedean place \(v\),
the restriction of \(\chi_{-,v}\) to \(F_v^\times=\bbR^\times\) is the sign character,
so every component of \(\ell_\infty(\chi_-)\) is odd.
Hence
\[
  \eta_{m+1}-\eta_m=\frac{1+\ell_\infty(\chi_-)}{2}\in\bbZ^{\bfa},
\]
and therefore the standing assumption \(\eta_{m+1}\in\bbZ^{\bfa}\) already implies \(\eta_m\in\bbZ^{\bfa}\).

Fix \(S\in F_{>0}=\Her_1(F)_{>0}\).
Let \(\overline{\calS}(M_{1,m}(\bbA_{E,\bfh}))\) be the conjugate complex vector space of \(\calS(M_{1,m}(\bbA_{E,\bfh}))\),
and let \(\overline{\omega}_{S,\bfh}\) be the conjugate Weil representation on it.
For notational simplicity,
we use \(\Phi\) both for a Schwartz function and for the corresponding vector in this conjugate space.
Set \(\widetilde V_{m,S}(h)=\Pi_{m+1}(h)\otimes\overline{\omega}_{S,\bfh}\).
For \(\phi\in \Pi_{m+1}(h)\) and \(\Phi\in\calS(M_{1,m}(\bbA_{E,\bfh}))\),
define
\begin{align*}
  \widetilde I_{m,S}^{\kappa}(h;\phi\otimes\Phi)(g)
   & = \calT_{S,\Phi}\!\left(I_{m+1}^{\kappa}(h;\phi)\right)(g) \\
   & = \int_{U_{m+1,1}(F)\backslash U_{m+1,1}(\bbA)}
  I_{m+1}^{\kappa}(h;\phi)(u\,\iota_{m+1,1}(g))\,
  \overline{\Theta_S(u,g;\Phi_S)}\,du
\end{align*}
for \(g\in G_m(\bbA)\).
The integral is linear in \(\phi\) and conjugate-linear in \(\Phi\),
hence linear on the conjugate Weil representation.
It therefore defines a map
\[
  \widetilde I_{m,S}^{\kappa}(h;\,\cdot):\widetilde V_{m,S}(h)\longrightarrow\calS_m^{\ell_m,\eta_m}.
\]

For \(g_\bfh\in G_m(\bbA_\bfh)\),
define
\[
  g_\bfh\cdot(\phi\otimes\Phi)
  =\Pi_{m+1}(h)(\iota_{m+1,1}(g_\bfh))\phi
  \otimes\overline{\omega}_{S,\bfh}(g_\bfh)\Phi.
\]

\subsection{Fourier coefficients and nonvanishing}

For \(a=m\),
abbreviate
\[
  d\mu_{S,v}=d\mu_{S,v}^{(m)},
  \ \text{and}\
  d\mu_{S,\bfh}=\prod_{v\in\bfh}d\mu_{S,v}.
\]
For a measurable subset \(X\subset M_{1,m}(\bbA_{E,\bfh})\),
write
\[
  \vol_{S,\bfh}(X):=\int_X d\mu_{S,\bfh}.
\]

\begin{lem}
  \label{lem:archimedean-FJ-factor}
  For \(w\in\widetilde V_{m,S}(h)\) and \(g_\bfh\in G_m(\bbA_\bfh)\),
  \begin{equation}
    \label{eq:even-Fourier-expansion-J}
    \widetilde I_{m,S}^{\kappa}(h;w)_{g_\bfh}(Z)
    =
    c_{m,S,\infty}
    \sum_{T\in\Her_m(F)_{>0}}
    |\det T|_\infty^{(\ell_m+1)/2}\,
    \calJ_T^S(g_\bfh\cdot w)
    e_\infty(\tr(TZ)),
  \end{equation}
  where \(c_{m,S,\infty}\in\bbC^\times\) is a nonzero constant depending only on \(m\),
  \(S\),
  and the fixed archimedean normalizations.
  For \(T\in\Her_m(F)_{>0}\),
  the functional \(\calJ_T^S\) is given on pure tensors by
  \begin{equation}
    \label{eq:global-finite-FJ-functional}
    \calJ_T^S(\phi\otimes\Phi)
    =
    \int_{M_{1,m}(\bbA_{E,\bfh})}
    \overline{\Phi(x)}
    \calW_{S\oplus T}
    (\Pi_{m+1}(h)(u_{m+1,1}(x,0,0))\phi) d\mu_{S,\bfh}(x).
  \end{equation}
  The integral is absolutely convergent and defines a linear functional on \(\widetilde V_{m,S}(h)\).
\end{lem}

\begin{proof}
  By linearity,
  it suffices to take \(w=\phi\otimes\Phi\).
  Specializing~\cite[Proposition~10.4]{Yamana2020HilbertHermitian} to \(i=1\),
  there is a constant \(c_{m,S,\infty}^{\circ}\in\bbC^\times\),
  independent of \(T\),
  \(\phi\),
  \(\Phi\),
  and \(g_\bfh\),
  such that
  \begin{equation}
    \label{eq:archimedean-FJ-coefficient}
    a(\widetilde I_{m,S}^{\kappa}(h;\phi\otimes\Phi)_{g_\bfh},T)
    =
    c_{m,S,\infty}^{\circ}
    |\det T|_\infty^{1/2}
    \calJ_T^S(g_\bfh\cdot(\phi\otimes\Phi))
    \prod_{v\in\bfa}\calA_v(S_v,T_v).
  \end{equation}
  Here
  \[
    \calA_v(S_v,T_v)
    =
    S_v^{\ell_{m+1,v}/2}
    |\det T_v|^{\ell_{m,v}/2}
    e^{-2\pi S_v}
    \int_{\bbC^m}
    e^{-4\pi S_vxx^*}\,d\mu_{S,v}(x).
  \]
  After identifying \(E_v\simeq\bbC\),
  \(d\mu_{S,v}\) is the self-dual measure for
  \[
    (x,y)\longmapsto
    \exp(2\pi\sqrt{-1}\tr_{\bbC/\bbR}(S_vxy^*) ).
  \]
  If \(dx_v\) denotes the standard Lebesgue measure on \(\bbC^m\),
  then
  \[
    d\mu_{S,v}(x)=(2S_v)^m\,dx_v,
  \]
  and hence
  \[
    \int_{\bbC^m}
    e^{-4\pi S_vxx^*}\,d\mu_{S,v}(x)
    =
    (2S_v)^m(4S_v)^{-m}
    =
    2^{-m}.
  \]
  Therefore
  \[
    \calA_v(S_v,T_v)
    =
    2^{-m}
    S_v^{\ell_{m+1,v}/2}
    e^{-2\pi S_v}
    |\det T_v|^{\ell_{m,v}/2}.
  \]
  Since \(\ell_{m+1,v}=\ell_{m,v}+1\),
  substituting this into \eqref{eq:archimedean-FJ-coefficient} and taking the product over \(v\in\bfa\) shows that there is a constant \(c_{m,S,\infty}\in\bbC^\times\),
  independent of \(T\),
  \(w\),
  and \(g_\bfh\),
  such that
  \[
    a(\widetilde I_{m,S}^{\kappa}(h;w)_{g_\bfh},T)
    =
    c_{m,S,\infty}
    |\det T|_\infty^{(\ell_m+1)/2}
    \calJ_T^S(g_\bfh\cdot w).
  \]
  This is equivalent to \eqref{eq:even-Fourier-expansion-J}.
\end{proof}

Every form in the image of \(\widetilde I_{m,S}^{\kappa}(h;\,\cdot)\) is holomorphic and cuspidal,
so its classical components have only positive-definite Fourier indices.
Since \(c_{m,S,\infty}\neq0\),
Lemma~\ref{lem:archimedean-FJ-factor} gives
\[
  \ker(\widetilde I_{m,S}^{\kappa}(h;\,\cdot))
  =
  \left\{w\in\widetilde V_{m,S}(h)\ \middle|\
  \calJ_T^S(g_\bfh\cdot w)=0
  \text{ for all }g_\bfh\in G_m(\bbA_\bfh),\
  T\in\Her_m(F)_{>0}\right\}.
\]
In particular,
\(\ker(\widetilde I_{m,S}^{\kappa}(h;\,\cdot))\) is \(G_m(\bbA_\bfh)\)-stable.

Define
\[
  \Pi_{m,S}(h)
  =
  \widetilde V_{m,S}(h)/
  \ker(\widetilde I_{m,S}^{\kappa}(h;\,\cdot)).
\]
The \(G_m(\bbA_\bfh)\)-action and the functionals \(\calJ_T^S\) descend to \(\Pi_{m,S}(h)\);
we use the same notation for the induced action and functionals.
The map \(\widetilde I_{m,S}^{\kappa}(h;\,\cdot)\) induces a map
\[
  I_{m,S}^{\kappa}(h;\,\cdot):
  \Pi_{m,S}(h)\longrightarrow
  \calS_m^{\ell_m,\eta_m}.
\]
For \(\phi_m\in \Pi_{m,S}(h)\),
we call \(I_{m,S}^{\kappa}(h;\phi_m)\) the even-degree Hermitian Ikeda lift associated with \(h\).

\begin{prop}
  \label{prop:even-nonvanishing}
  For every \(S\in F_{>0}\),
  the representation \(\Pi_{m,S}(h)\) is nonzero and
  \[
    I_{m,S}^{\kappa}(h;\,\cdot):
    \Pi_{m,S}(h)\hookrightarrow\calS_m^{\ell_m,\eta_m}
  \]
  is a nonzero \(G_m(\bbA_\bfh)\)-equivariant embedding.
  For \(\phi_m\in \Pi_{m,S}(h)\) and \(g_\bfh\in G_m(\bbA_\bfh)\),
  \[
    I_{m,S}^{\kappa}(h;\phi_m)_{g_\bfh}(Z)
    =
    c_{m,S,\infty}
    \sum_{T\in\Her_m(F)_{>0}}
    |\det T|_\infty^{(\ell_m+1)/2}\,
    \calJ_T^S(g_\bfh\cdot\phi_m)
    e_\infty(\tr(TZ)).
  \]
\end{prop}

\begin{proof}
  Equivariance,
  injectivity,
  and the Fourier expansion follow from the definition of \(\Pi_{m,S}(h)\) and Lemma~\ref{lem:archimedean-FJ-factor}.
  It remains to prove nonvanishing.
  Fix \(T\in\Her_m(F)_{>0}\).
  By \cite[Propositions~4.5(2) and~6.3]{Yamana2020HilbertHermitian},
  the functional \(\calW_{S\oplus T}\) is nonzero.
  Hence we may choose \(\phi\in\Pi_{m+1}(h)\) such that
  \[
    \calW_{S\oplus T}(\phi)\neq0.
  \]
  By smoothness,
  choose a compact open neighborhood \(X\subset M_{1,m}(\bbA_{E,\bfh})\) of \(0\) such that \(\phi\) is fixed by \(u_{m+1,1}(x,0,0)\) for every \(x\in X\).
  Taking \(\Phi=1_X\),
  we obtain directly from \eqref{eq:global-finite-FJ-functional}
  \[
    \calJ_T^S(\phi\otimes\Phi)
    =
    \vol_{S,\bfh}(X)\,
    \calW_{S\oplus T}(\phi)
    \neq0.
  \]
  Since \(c_{m,S,\infty}\neq0\),
  \eqref{eq:even-Fourier-expansion-J} shows that the image of \(\widetilde I_{m,S}^{\kappa}(h;\,\cdot)\) is nonzero.
\end{proof}

\section{Arthur parameters and global packets}
\label{sec:arthur-packets}

We recall the endoscopic classification,
define the global Arthur parameters attached to the lifts,
and identify the packets containing them from their unramified local parameters.

\subsection{Endoscopic classification}

For \(N\geq1\),
put \(G_N=\U_{N,N}\),
so that \(\widehat G_N=\GL_{2N}(\bbC)\).
Write \(W_F\) and \(W_E\) for the global Weil groups of \(F\) and \(E\).
At an unramified finite place \(v\),
put \(K_{N,v}=G_N(\calO_{F_v})\).
We fix the following normalization of the standard \(L\)-embedding:
\[
  \xi_{\std}:{}^L{G_N}\rightarrow
  {}^L{(\Res_{E/F}\GL_{2N})}
  =(\GL_{2N}(\bbC)\times\GL_{2N}(\bbC))\rtimes W_F.
\]
Let \(J_N\) be the matrix used above to define \(G_N\),
and choose \(w_c\in W_F\setminus W_E\) mapping to \(c\in\Gal(E/F)\).
The embedding is determined by
\[
  \begin{aligned}
    \xi_{\std}(g\rtimes1)
     & =(g,{}^tg^{-1})\rtimes1
     &                               & (g\in\GL_{2N}(\bbC)), \\
    \xi_{\std}(1_{2N}\rtimes\gamma)
     & =(1_{2N},1_{2N})\rtimes\gamma
     &                               & (\gamma\in W_E),      \\
    \xi_{\std}(1_{2N}\rtimes w_c)
     & =(J_N,J_N^{-1})\rtimes w_c.
  \end{aligned}
\]
Here \(W_F\) acts on the dual groups through \(\Gal(E/F)\);
on \(\GL_{2N}(\bbC)\times\GL_{2N}(\bbC)\),
the nontrivial Galois element interchanges the two factors.

We write an elliptic global Arthur parameter for \(G_N\) in the form
\[
  \psi=\mathop{\boxplus}\limits_{i=1}^t\mu_i[d_i].
\]
Here \(\sum_{i=1}^t m_id_i=2N\),
the \(\mu_i\) are pairwise distinct conjugate self-dual unitary cuspidal representations of \(\GL_{m_i}(\bbA_E)\),
and \([d_i]\) denotes the \(d_i\)-dimensional irreducible representation of the Arthur \(\SL_2(\bbC)\).
For a conjugate self-dual cuspidal representation \(\mu\) of \(\GL_m(\bbA_E)\),
write \(L(s,\mu,\Asai^\pm)\) for its two Asai \(L\)-functions.
We write \(\delta(\mu)\in\{\pm1\}\) for the sign characterized by the condition that \(L\left(s,\mu,\Asai^{(-1)^{m-1}\delta(\mu)}\right)\) has a pole at \(s=1\).

By~\cite[Remark~2.4.6, Definition~2.4.7, and Theorem~2.5.4(a)]{Mok2015Endoscopic},
the formal sum above factors through the standard embedding \(\xi_{\std}\) precisely when
\begin{equation}\label{eq:mok-sign}
  \delta(\mu_i)=(-1)^{m_i+d_i+1}
\end{equation}
for every \(1\leq i\leq t\).
When this holds,
we write \(\psi\in\Psi_2(G_N,\xi_{\std})\).
For each place \(v\),
write \(\psi_v\) for the localization and \(\Pi_{\psi_v}\) for the corresponding local packet.
The global packet is the restricted product
\[
  \Pi_\psi
  =
  \left\{
  \pi={\bigotimes_v}'\pi_v
  \ \middle|\
  \pi_v\in\Pi_{\psi_v}
  \right\}.
\]
Following Mok,
we use the notation \(\calS_\psi\),
\(\epsilon_\psi\),
and \(\langle\,\cdot\,,\pi\rangle_\psi\) for the global component group,
its canonical sign character,
and the global pairing;
see~\cite[\S\S~2.4--2.5]{Mok2015Endoscopic}.
We write \(L^2_{\disc}(G_N(F)\backslash G_N(\bbA))\) for the discrete part of the automorphic \(L^2\)-spectrum.
For an automorphic representation \(\pi\) of \(G_N(\bbA)\),
\(L(s,\pi,\Std)\) denotes the standard \(L\)-function attached to \(\xi_{\std}\).

\begin{prop}
  \label{prop:mok-global-multiplicity}
  For \(\psi\in\Psi_2(G_N,\xi_{\std})\),
  put
  \[
    \Pi_\psi(\epsilon_\psi)
    =
    \{\pi\in\Pi_\psi:\langle\,\cdot\,,\pi\rangle_\psi=\epsilon_\psi\}.
  \]
  Then
  \[
    L^2_{\disc}(G_N(F)\backslash G_N(\bbA))
    \simeq
    \bigoplus_{\psi\in\Psi_2(G_N,\xi_{\std})}
    \ \bigoplus_{\pi\in\Pi_\psi(\epsilon_\psi)}\pi.
  \]
\end{prop}

\begin{lem}
  \label{lem:mok-spherical-packet}
  Let \(v\) be a finite place at which \(G_N\) is unramified,
  and let \(\psi_v\) be a local Arthur parameter.
  If \(\Pi_{\psi_v}\) contains a \(K_{N,v}\)-spherical representation \(\pi_v^{\sph}\),
  then \(\pi_v^{\sph}\) is the unique spherical member of the packet,
  and its character on \(\calS_{\psi_v}\) is trivial.
\end{lem}

\begin{proof}
  This follows from~\cite[Theorem~2.5.1(a) and \S~7.6]{Mok2015Endoscopic} together with~\cite{Shahidi1981Certain,Shahidi1990Plancherel}.
\end{proof}

\subsection{Arthur parameters of the lifts}
\label{subsec:arthur-parameters-lifts}

We say that \(h\) has CM by \(E\) if \(\pi_h\simeq\pi_h\otimes\chi_{E/F}\).
If \(h\) has CM by \(E\),
quadratic automorphic induction \cite[Chapter~3, Theorems~4.2(b) and~6.2]{ArthurClozel1989BaseChange} gives a Hecke character \(\theta\) of \(E\),
unique up to conjugation,
such that \(\pi_h\simeq\AI_{E/F}(\theta)\),
with \(\theta\neq\theta^c\).

For \(n\geq1\),
put
\[
  \psi_n(h)=
  \begin{cases}
    (((\vartheta_n)^c)^{-1}\otimes\BC_{E/F}(\pi_h))[n]                           & \text{if }h\text{ does not have CM by }E, \\
    (((\vartheta_n)^c)^{-1}\theta)[n]\boxplus(((\vartheta_n)^c)^{-1}\theta^c)[n] & \text{if }\pi_h\simeq\AI_{E/F}(\theta).
  \end{cases}
\]
Here \(\BC_{E/F}\) denotes quadratic base change from \(F\) to \(E\),
and \(\AI_{E/F}\) denotes automorphic induction from \(E\) to \(F\).

\begin{prop}\label{prop:global-parameter-all-cases}
  For every \(n\geq1\),
  \[
    \psi_n(h)\in\Psi_2(G_n,\xi_{\std}).
  \]
\end{prop}

\begin{proof}
  Suppose first that \(h\) does not have CM by \(E\).
  By cyclic base change~\cite[Chapter~3, Theorem~4.2(a)]{ArthurClozel1989BaseChange},
  \(\BC_{E/F}(\pi_h)\) is cuspidal.
  Put
  \[
    \mu_n=((\vartheta_n)^c)^{-1}\otimes\BC_{E/F}(\pi_h).
  \]
  The identities \(\vartheta_n\vartheta_n^c=\omega_h\circ N_{E/F}\) and \(\pi_h^\vee\simeq\pi_h\otimes\omega_h^{-1}\) imply that \(\mu_n\) is conjugate self-dual.
  By~\cite[Lemma~C.2]{Yamana2020HilbertHermitian},
  \(L(s,\mu_1,\Asai^-)\) has a pole at \(s=1\).
  For even \(n\),
  we have \(\mu_n=\mu_1\otimes\chi_-^c\),
  and this twist interchanges the Asai signs since \(\chi_-^c|_{\bbA^\times}=\chi_{E/F}\).
  Hence \(\delta(\mu_n)=(-1)^{n+1}\),
  which is exactly~\eqref{eq:mok-sign} for \(m_i=2\) and \(d_i=n\).

  Now suppose that \(\pi_h\simeq\AI_{E/F}(\theta)\),
  with \(\theta\neq\theta^c\).
  Then \(\BC_{E/F}(\pi_h)=\theta\boxplus\theta^c\) and \(\theta\theta^c=\omega_h\circ N_{E/F}\).
  Hence the two characters \(((\vartheta_n)^c)^{-1}\theta\) and \(((\vartheta_n)^c)^{-1}\theta^c\) are conjugate self-dual.
  Their restriction to \(\bbA^\times\) is \(\chi_{E/F}\) when \(n\) is odd and trivial when \(n\) is even.
  Hence both characters have \(\delta=(-1)^n\),
  which is~\eqref{eq:mok-sign} for \(m_i=1\) and \(d_i=n\).
\end{proof}

\subsection{Unramified local parameters}
\label{subsec:common-local-setup}

Fix a finite place \(v\in\bfh\).
Write \(W_{F_v}\) and \(W_{E_v}\) for the corresponding local Weil groups,
and write \(\BC_{E_v/F_v}\) for local quadratic base change.
We recall the local realizations and functionals used in both odd and even degree.
For a quasi-character \(\zeta:E_v^\times\rightarrow\bbC^\times\),
let \(I_{n,v}(\zeta)\) be the normalized degenerate principal series of \(G_n(F_v)\) consisting of smooth functions satisfying
\[
  f(m_n(A)n_n(Z)g)
  =
  \zeta(\det A)
  |\det A|_{E_v}^{n/2}
  f(g)
\]
for \(A\in\GL_n(E_v)\),
\(Z\in\Her_n(F_v)\),
and \(g\in G_n(F_v)\).
Put \(K_{n,v}=G_n(\calO_{F_v})\).

For nondegenerate \(B\in\Her_n(F_v)\),
let
\[
  w_{B,v}^{\zeta}:I_{n,v}(\zeta)\longrightarrow\bbC
\]
be the normalized degenerate Whittaker functional attached to \(B\) in~\cite[Definition~5.1]{Yamana2020HilbertHermitian}.
Put
\[
  b_{n,v}(\zeta)
  =
  \prod_{j=1}^{n}
  L\left(
  j,
  \left.\zeta\right|_{F_v^\times}\chi_v^{\,n+j}
  \right).
\]
For \(f\in I_{n,v}(\zeta)\),
let \(f_s\in I_{n,v}(\zeta|\cdot|_{E_v}^s)\) be the standard holomorphic section through \(f\).
Then
\begin{equation}
  \label{eq:local-Jacquet-functional}
  w_{B,v}^{\zeta}(f)
  =
  |\det B|_{F_v}^{n/2}b_{n,v}(\zeta)
  \left.
  \int_{\Her_n(F_v)}
  f_s(J_n n_n(Z))
  \psi_v(\tr(BZ))\,dZ
  \right|_{s=0}.
\end{equation}
The integral is initially convergent for \(\Re(s)\gg0\),
and the normalized expression is holomorphic at \(s=0\).

We apply the preceding local constructions at unramified places to identify the global packets of the odd- and even-degree lifts.

For an irreducible admissible representation \(\tau\) of a local general linear or unitary group,
write \(\phi_\tau\) for its Langlands parameter under the local Langlands correspondence.

\begin{dfn}
  A finite place \(v\in\bfh\) is called \emph{good} if
  \begin{itemize}
    \item \(v\) is inert or split in \(E/F\) and \(\psi_v\) has conductor \(\calO_{F_v}\);
    \item \(\pi_{h,v}\) and \(\vartheta_v^+\) are unramified.
  \end{itemize}
\end{dfn}

Fix a good finite place \(v\).
For the even-degree calculation in this subsection,
we further assume that \(\chi_{-,v}\) is unramified and \(S\in\calO_{F_v}^\times\).
Under these assumptions,
we identify the spherical representation obtained from the intertwining map below with the unramified member of the packet attached to \(\psi_{m,v}(h)\).

Since \(\pi_{h,v}\) is unramified,
choose an unramified quasi-character \(\mu_v:F_v^\times\rightarrow\bbC^\times\) such that
\[
  \phi_{\pi_{h,v}}
  \simeq
  \mu_v\oplus\omega_{h,v}\mu_v^{-1}.
\]
Define \(\zeta_v^+=(\vartheta_v^+)^{-c}(\mu_v\circ N_{E_v/F_v})\) and \(\zeta_v^-=\zeta_v^+\chi_{-,v}^{-1}\).
For every degree \(r\geq1\),
put
\[
  \zeta_{r,v}=
  \begin{cases}
    \zeta_v^+ & (r\text{ odd}),  \\
    \zeta_v^- & (r\text{ even}).
  \end{cases}
\]
We use this convention throughout whenever \(\zeta_v^\pm\) are defined,
including at ramified places.
The identity
\[
  \vartheta_v^+(\vartheta_v^+)^c
  =
  \omega_{h,v}\circ N_{E_v/F_v}
\]
shows that the local parameter of \(((\vartheta_{n,v})^c)^{-1}\otimes\BC_{E_v/F_v}(\pi_{h,v})\) is \(\zeta_{n,v}\oplus(\zeta_{n,v})^{-c}\).

Write \(\psi_{n,v}(h)\) for the localization of \(\psi_n(h)\),
\(\Pi_{\psi_{n,v}(h)}\) for its local Arthur packet,
and \(\phi_{n,v}(h)\) for the associated Langlands parameter.
Thus,
on \(W_{E_v}=\prod_{w\mid v}W_{E_w}\),
we have
\begin{equation}
  \label{eq:local-psi-parameter}
  \left.\phi_{n,v}(h)\right|_{W_{E_v}}
  \simeq
  \left(
  \zeta_{n,v}
  \oplus
    (\zeta_{n,v})^{-c}
  \right)
  \otimes
  \bigoplus_{j=0}^{n-1}
  |\,\cdot\,|_{E_v}^{(n-1)/2-j}.
\end{equation}

\begin{lem}
  \label{lem:unramified-Whittaker-normalization}
  Let \(v\) be good,
  let \(\zeta:E_v^\times\rightarrow\bbC^\times\) be unramified,
  and let
  \[
    \phi^\circ
    \in
    I_{n,v}(\zeta)^{K_{n,v}}
  \]
  be the normalized spherical section with \(\phi^\circ(1)=1\).
  If \(B\in\Her_n(F_v)\) is integral and unimodular,
  then
  \[
    w_{B,v}^{\zeta}(\phi^\circ)=1.
  \]
\end{lem}

\begin{proof}
  If \(E_v\) is a field,
  the Hermitian Siegel series polynomial attached to \(B\) is \(1\) by~\cite[\S~2]{Ikeda2008Lifting},
  and the assertion follows from \cite[Lemma~9.2]{Yamana2020HilbertHermitian}.
  If \(E_v=F_v\oplus F_v\),
  the same calculation applies after removing the unramified central twist in the split realization of \cite[Appendix~B]{Yamana2020HilbertHermitian}.
  This twist is trivial on the elements \(J_n n_n(Z)\) occurring in \eqref{eq:local-Jacquet-functional}.
\end{proof}

By \cite[Propositions~6.2--6.3 and Corollary~6.4]{Yamana2020HilbertHermitian},
the local representation occurring in the odd-degree lift is realized in \(I_{m+1,v}(\zeta_v^+)\),
with its Shalika functional identified with \(w_{B,v}^{\zeta_v^+}\).

To pass from degree \(m+1\) to degree \(m\),
we use the intertwining map of \cite[Lemma~10.1]{Yamana2020HilbertHermitian}.
Recall the subgroup \(Y_{m+1,1}(F_v)\) defined above,
and put
\[
  \eta_{m+1,1}
  =
  \begin{pmatrix}
    0 & 0   & -1 & 0   \\
    0 & 1_m & 0  & 0   \\
    1 & 0   & 0  & 0   \\
    0 & 0   & 0  & 1_m
  \end{pmatrix}.
\]
For \(f\in I_{m+1,v}(\zeta_v^+)\),
\(\Phi\in\calS(M_{1,m}(E_v))\),
and \(g\in G_m(F_v)\),
define
\begin{equation}
  \label{eq:local-FJ-integral}
  \beta_{S,v}^{\chi_-}(g;f\otimes\Phi)
  =
  L\left(
  m+1,
  \left.\zeta_v^+\right|_{F_v^\times}
  \right)
  \int_{Y_{m+1,1}(F_v)\backslash U_{m+1,1}(F_v)}
  f(\eta_{m+1,1}u\,\iota_{m+1,1}(g))
  \overline{
    (\omega_{S_v}(u,g)\Phi)(0)
  }\,
  du.
\end{equation}
Here the quotient measure is the one fixed above.
By \cite[Lemma~10.1]{Yamana2020HilbertHermitian},
this defines a \(G_m(F_v)\)-intertwining map
\[
  \beta_{S,v}^{\chi_-}:
  I_{m+1,v}(\zeta_v^+)
  \otimes
  \overline{\omega_{S_v}}
  \longrightarrow
  I_{m,v}(\zeta_v^-).
\]

Let \(\phi_v^\circ \in I_{m+1,v}(\zeta_v^+)^{K_{m+1,v}}\) be the normalized spherical section with \(\phi_v^\circ(1)=1\),
and put \(\Phi_v^\circ=1_{M_{1,m}(\calO_{E_v})}\).
Define
\[
  f_v^\circ(g)
  =
  \beta_{S,v}^{\chi_-}
  (g;\phi_v^\circ\otimes\Phi_v^\circ).
\]

\begin{prop}
  \label{prop:local-satake-even}
  Let \(v\) be good,
  and assume that \(\chi_{-,v}\) is unramified and \(S\in\calO_{F_v}^\times\).
  Then
  \[
    0\neq f_v^\circ
    \in
    I_{m,v}(\zeta_v^-)^{K_{m,v}}.
  \]
  The \(G_m(F_v)\)-subrepresentation of \(I_{m,v}(\zeta_v^-)\) generated by \(f_v^\circ\) has,
  up to isomorphism,
  a unique irreducible quotient \(\tau_v^\circ\),
  and
  \[
    \left.
    (\xi_{\std}\circ\phi_{\tau_v^\circ})
    \right|_{W_{E_v}}
    \simeq
    \left.
    \phi_{m,v}(h)
    \right|_{W_{E_v}}.
  \]
  Hence \(\tau_v^\circ\) is the unique spherical member of \(\Pi_{\psi_{m,v}(h)}\).
\end{prop}

\begin{proof}
  By the Whittaker identity in \cite[Lemma~10.1]{Yamana2020HilbertHermitian},
  there is \(C_{S,v}\in\bbC^\times\) such that
  \begin{align*}
    w_{1_m,v}^{\zeta_v^-}(f_v^\circ)
     & =
    C_{S,v}
    \int_{M_{1,m}(E_v)}
    \overline{\Phi_v^\circ(x)}\,
    w_{S\oplus1_m,v}^{\zeta_v^+}
    \left(
    I_{m+1,v}(\zeta_v^+)
    (u_{m+1,1}(x,0,0))\phi_v^\circ
    \right)\,dx
    \\
     & =
    C_{S,v}
    \vol(M_{1,m}(\calO_{E_v}))
    w_{S\oplus1_m,v}^{\zeta_v^+}(\phi_v^\circ),
  \end{align*}
  since \(u_{m+1,1}(x,0,0)\in K_{m+1,v}\) for \(x\in M_{1,m}(\calO_{E_v})\).

  Since \(S\oplus1_m\) is integral and unimodular,
  Lemma~\ref{lem:unramified-Whittaker-normalization} gives
  \[
    w_{S\oplus1_m,v}^{\zeta_v^+}(\phi_v^\circ)=1.
  \]

  It follows that
  \[
    w_{1_m,v}^{\zeta_v^-}(f_v^\circ)\neq0,
  \]
  and hence \(f_v^\circ\neq0\).

  Since \(\phi_v^\circ\) and \(\Phi_v^\circ\) are spherical,
  \(f_v^\circ\) is \(K_{m,v}\)-fixed.
  As
  \[
    \dim I_{m,v}(\zeta_v^-)^{K_{m,v}}=1,
  \]
  every irreducible quotient of the subrepresentation generated by \(f_v^\circ\) is spherical with the same Satake parameter.
  Thus this subrepresentation has a unique irreducible quotient \(\tau_v^\circ\) up to isomorphism.

  Since \(\tau_v^\circ\) is spherical,
  \[
    \left.
    (\xi_{\std}\circ\phi_{\tau_v^\circ})
    \right|_{W_{E_v}}
    \simeq
    \left(
    \zeta_v^-\oplus(\zeta_v^-)^{-c}
    \right)
    \otimes
    \bigoplus_{j=0}^{m-1}
    |\,\cdot\,|_{E_v}^{(m-1)/2-j}.
  \]
  By~\eqref{eq:local-psi-parameter},
  this equals \(\left.\phi_{m,v}(h)\right|_{W_{E_v}}\).
  The last assertion follows from Lemma~\ref{lem:mok-spherical-packet}.
\end{proof}

\subsection{Identification of the global packets}

For odd \(r\),
let \(\calV_r(h)\) be the \(G_r(\bbA)\)-span of \(\Im(I_r^\kappa(h;\,\cdot))\) in the cuspidal spectrum.
For even \(m\),
define \(\calV_{m,S}(h)\) similarly using \(\Im(I_{m,S}^\kappa(h;\,\cdot))\).

\begin{cor}
  \label{cor:global-packet-identification}
  \begin{enumerate}
    \item If \(r\) is odd,
          every irreducible constituent \(\rho\) of \(\calV_r(h)\) belongs to \(\Pi_{\psi_r(h)}\),
          and
          \[
            L(s,\rho,\Std)
            =
            \prod_{j=0}^{r-1}
            L\left(
            s+\frac{r-1}{2}-j,
            ((\vartheta_r)^c)^{-1}\otimes\BC_{E/F}(\pi_h)
            \right).
          \]

    \item If \(m\) is even,
          every irreducible constituent \(\sigma\) of \(\calV_{m,S}(h)\) belongs to \(\Pi_{\psi_m(h)}\),
          and
          \[
            L(s,\sigma,\Std)
            =
            \prod_{j=0}^{m-1}
            L\left(
            s+\frac{m-1}{2}-j,
            ((\vartheta_m)^c)^{-1}\otimes\BC_{E/F}(\pi_h)
            \right).
          \]
  \end{enumerate}
\end{cor}
\begin{proof}
  At almost every finite place,
  the local parameters are those attached to \(\psi_r(h)\) in odd degree by \cite[Theorem~1.1 and Appendix~A]{Yamana2020HilbertHermitian},
  and to \(\psi_m(h)\) in even degree by Proposition~\ref{prop:local-satake-even}.
  Strong multiplicity one for \(\GL_n\) \cite{JacquetShalika1981EulerII},
  together with the uniqueness of the isobaric decomposition,
  identifies the corresponding global isobaric representations with those attached to \(\psi_r(h)\) and \(\psi_m(h)\).
  Proposition~\ref{prop:mok-global-multiplicity} then gives the packet assertions.
  The formulas for the standard \(L\)-functions follow from the definitions of \(\psi_r(h)\) and \(\psi_m(h)\).
\end{proof}
\section{Local constituents and Fourier--Jacobi descent}
\label{sec:local-fj-descent}

Having identified the global packets,
we now determine the local constituents of the odd-degree lift and their descent to even degree.
The local analysis has three parts:
degenerate principal series at nonsplit places \(v\) where \(\pi_{h,v}\) is not supercuspidal,
theta correspondence at places \(v\) where \(\pi_{h,v}\) is dihedral supercuspidal,
and the remaining cases,
where the relevant local packet is a singleton.
For the global data fixed above,
\(m\) is even and \(S_v\) denotes the image of \(S\) in \(F_v^\times\).

\begin{dfn}\label{def:local-FJ-module}
  For a smooth representation \(\rho_v\) of \(G_{m+1}(F_v)\),
  define its local Fourier--Jacobi module of index \(S_v\) by
  \[
    \FJ_{S_v}(\rho_v)
    =\left(\rho_v\otimes\overline{\omega_{S_v}}\right)_{U_{m+1,1}(F_v)}.
  \]
  Here \(\overline{\omega_{S_v}}\) is the complex-conjugate Weil representation,
  \(U_{m+1,1}(F_v)\) acts diagonally,
  and the subscript denotes the maximal quotient on which this diagonal action is trivial.
\end{dfn}

If \(\rho_v\) is a \(G_{m+1}(F_v)\)-subrepresentation of \(I_{m+1,v}(\zeta_{m+1,v})\),
the restriction of \(\beta_{S,v}^{\chi_-}\) in~\eqref{eq:local-FJ-integral} is invariant under the diagonal action of \(U_{m+1,1}(F_v)\).
It therefore factors through a \(G_m(F_v)\)-equivariant map
\[
  \FJ_{S_v}(\rho_v)\longrightarrow I_{m,v}(\zeta_{m,v}).
\]
Hence the images computed below are quotients of the corresponding Fourier--Jacobi modules.

\subsection{Restriction to the unitary group}

Let \(r\) be odd and write
\[
  \Pi_r(h)={\bigotimes_{v\in\bfh}}'\Pi_{r,v}(h).
\]
For a finite place \(v\),
put \(\widetilde G_{r,v}=\GU_{r,r}(F_v)\),
\(G_{r,v}=\U_{r,r}(F_v)\),
and \(H_{r,v}=\widetilde Z_{r,v}G_{r,v}\),
where \(\widetilde Z_{r,v}\) is the center of \(\widetilde G_{r,v}\).
Since
\[
  \widetilde G_{r,v}/H_{r,v}
  \simeq
  F_v^\times/N_{E_v/F_v}(E_v^\times),
\]
Clifford theory shows that \(\Pi_{r,v}(h)|_{G_{r,v}}\) is multiplicity-free and has at most two irreducible constituents.

At a good place \(v\),
let \(\phi_{r,v}^\circ\) be the spherical vector of \(\Pi_{r,v}(h)\) corresponding to the section with value \(1\) at the identity in the degenerate principal-series realization of~\cite[Proposition~6.3]{Yamana2020HilbertHermitian}.
If \(B_v\in\Her_r(F_v)\) is integral and unimodular,
then
\begin{equation}
  \label{eq:spherical-Whittaker-normalization}
  \calW_{B_v,v}(\phi_{r,v}^\circ)=1.
\end{equation}
Indeed,
the comparison factor in~\cite[Corollary~6.4]{Yamana2020HilbertHermitian} is
\[
  \omega_{h,v}(-1)^{(r-1)/2}
  \mu_v((-1)^{(r-1)/2}\det B_v)^{-1}=1,
\]
since \(\omega_{h,v}\) and \(\mu_v\) are unramified and \(\det B_v\in\calO_{F_v}^\times\).
Thus \(\calW_{B_v,v}\) is identified with \(w_{B_v,v}^{\zeta_{r,v}}\) on this realization of \(\Pi_{r,v}(h)\).
The assertion therefore follows from Lemma~\ref{lem:unramified-Whittaker-normalization}.

For nondegenerate \(B\in\Her_r(F_v)\),
put
\[
  \epsilon_{r,v}(B)
  =\chi_v\left((-1)^{r(r-1)/2}\det B\right).
\]
For a smooth representation \(\rho\) of \(G_{r,v}\),
put
\[
  \Wh_{B,v}(\rho)
  =\Hom_{N_r(F_v)}(\rho,\psi_{B,v}),
\]
and
\[
  \psi_{B,v}(n_r(X))=\psi_v(\tr(BX)).
\]

We first describe the restriction in the supercuspidal case.

\begin{lem}
  \label{lem:odd-supercuspidal-properties}
  Let \(r\) be odd and let \(v\) be a nonsplit finite place such that \(\pi_{h,v}\) is supercuspidal.
  \begin{enumerate}
    \item If \(\pi_{h,v}\not\simeq\pi_{h,v}\otimes\chi_v\),
          then \(\left.\Pi_{r,v}(h)\right|_{G_{r,v}}\) is irreducible and is the unique member of \(\Pi_{\psi_{r,v}(h)}\).
    \item If \(\pi_{h,v}\simeq\pi_{h,v}\otimes\chi_v\),
          then \(\left.\Pi_{r,v}(h)\right|_{G_{r,v}}\) is the direct sum of two nonisomorphic irreducible members of \(\Pi_{\psi_{r,v}(h)}\).
  \end{enumerate}
  In both cases,
  every irreducible constituent \(\rho\) of the restriction satisfies
  \[
    \Wh_{B,v}(\rho)\neq0
  \]
  for some nondegenerate \(B\in\Her_r(F_v)\).
  In the first case,
  this holds for every nondegenerate \(B\in\Her_r(F_v)\).
\end{lem}

\begin{proof}
  Fix an irreducible constituent \(\rho\) of \(\left.\Pi_{r,v}(h)\right|_{G_{r,v}}\).
  The degenerate Whittaker functional in~\cite[Proposition~4.5(2)]{Yamana2020HilbertHermitian} is nonzero and satisfies the Shalika equivariance in~\cite[Definition~3.1 and Proposition~7.2]{Yamana2020HilbertHermitian}.
  Together with the action of \(\widetilde G_{r,v}\) on the constituents of the restriction,
  this gives a nondegenerate \(B_v\in\Her_r(F_v)\) such that
  \[
    \Wh_{B_v,v}(\rho)\neq0.
  \]

  We next show that \(\rho\) belongs to the asserted packet.
  By weak approximation,
  choose \(B\in\Her_r(F)_{>0}\) whose localization at \(v\) lies in the \(\GL_r(E_v)\)-orbit of \(B_v\).
  The factorization of the global Shalika functional gives an irreducible constituent of \(\left.\Pi_r(h)\right|_{G_r(\bbA_\bfh)}\) whose local component at \(v\) is \(\rho\) and which contains a vector \(\phi\) satisfying
  \[
    \calW_B(\phi)\neq0.
  \]
  Hence \(I_r^\kappa(h;\phi)\neq0\),
  so this constituent occurs in the finite part of the odd-degree lift.
  Corollary~\ref{cor:global-packet-identification} now implies \(\rho\in\Pi_{\psi_{r,v}(h)}\).

  Suppose first that \(\pi_{h,v}\not\simeq\pi_{h,v}\otimes\chi_v\).
  The parameter \(\psi_{r,v}(h)\) consists of a single simple summand attached to a supercuspidal representation of \(\GL_2(E_v)\),
  and its component group is trivial.
  Its packet therefore has a single member by~\cite[Theorem~2.5.1(b), (8.1.1), and (8.2.7)]{Mok2015Endoscopic}.
  Since the restriction is multiplicity-free,
  it is irreducible and equals this member.
  As \(r\) is odd,
  all nondegenerate Hermitian forms of rank \(r\) over \(E_v/F_v\) lie in a single similitude orbit.
  Shalika equivariance therefore gives the nonvanishing for every nondegenerate \(B\).

  Finally,
  suppose that \(\pi_{h,v}\simeq\pi_{h,v}\otimes\chi_v\).
  The identity \(\pi_{h,v}\simeq\pi_{h,v}\otimes\chi_v\) implies
  \[
    \Pi_{r,v}(h)
    \simeq\Pi_{r,v}(h)\otimes(\chi_v\circ\nu).
  \]
  Clifford theory then shows that the restriction to \(G_{r,v}\) is the direct sum of two nonisomorphic irreducible representations.
\end{proof}

\subsection{Nonsupercuspidal representations}

We use the degenerate principal series defined in Subsection~\ref{subsec:common-local-setup}.
Consider the nonsplit places \(v\) for which \(\pi_{h,v}\) is not supercuspidal.
Let \(B_2\subset\GL_2\) be the upper triangular Borel subgroup.
Choose a quasi-character \(\mu_v:F_v^\times\rightarrow\bbC^\times\) such that \(\pi_{h,v}\) is the unique generic irreducible subquotient of
\[
  \Ind_{B_2(F_v)}^{\GL_2(F_v)}\left(\mu_v\boxtimes\omega_{h,v}\mu_v^{-1}\right),
\]
where,
when \(\pi_{h,v}\) is a twist of the Steinberg representation,
the two inducing characters are ordered so that
\begin{equation}
  \label{eq:steinberg-xi}
  \mu_v^2\omega_{h,v}^{-1}=|\cdot|_{F_v}.
\end{equation}
Put
\begin{align*}
  \zeta_{m+1,v}
   & = (\vartheta_v^+)^{-c}
  (\mu_v\circ N_{E_v/F_v}),                    \\
  \zeta_{m,v}
   & = \zeta_{m+1,v}\chi_{-,v}^{-1},           \\
  \xi_v
   & = \left.\zeta_{m+1,v}\right|_{F_v^\times}
  = \mu_v^2\omega_{h,v}^{-1}.
\end{align*}

\begin{prop}
  \label{prop:odd-nonsupercuspidal-structure}
  Let \(r\) be odd and let \(v\) be a nonsplit finite place such that \(\pi_{h,v}\) is not supercuspidal.
  Retain the notation \(\zeta_{r,v}\) and \(\xi_v\) introduced above.
  \begin{enumerate}
    \item If \(\xi_v\notin\{\chi_v,|\,\cdot\,|_{F_v}\}\),
          then
          \[
            \left.\Pi_{r,v}(h)\right|_{G_{r,v}}
            \simeq I_{r,v}(\zeta_{r,v}),
          \]
          and this representation is irreducible.

    \item If \(\xi_v=|\,\cdot\,|_{F_v}\),
          then
          \[
            \left.\Pi_{r,v}(h)\right|_{G_{r,v}}
            \simeq A_{r,v}^{+}(\zeta_{r,v}),
          \]
          where \(A_{r,v}^{+}(\zeta_{r,v})\) is the unique irreducible subrepresentation of \(I_{r,v}(\zeta_{r,v})\).
          Moreover,
          for every nondegenerate \(B\in\Her_r(F_v)\),
          \[
            \dim\Wh_{B,v}\left(A_{r,v}^{+}(\zeta_{r,v})\right)=1.
          \]

    \item If \(\xi_v=\chi_v\),
          then
          \[
            \left.\Pi_{r,v}(h)\right|_{G_{r,v}}
            \simeq I_{r,v}(\zeta_{r,v})
            \simeq A_{r,v}^{+}(\zeta_{r,v})\oplus A_{r,v}^{-}(\zeta_{r,v}).
          \]
          The two summands are irreducible.
          For \(\delta\in\{\pm1\}\) and every nondegenerate \(B\in\Her_r(F_v)\),
          their signs are determined by
          \[
            \dim\Wh_{B,v}\left(A_{r,v}^{\delta}(\zeta_{r,v})\right)
            =\begin{cases}
              1 & \text{if }\epsilon_{r,v}(B)=\delta, \\
              0 & \text{otherwise}.
            \end{cases}
          \]
  \end{enumerate}
\end{prop}

\begin{proof}
  By \cite[Sections~5--6 and Proposition~5.4]{Yamana2020HilbertHermitian},
  the restriction is realized as the unique irreducible subrepresentation of \(I_{r,v}(\zeta_{r,v})\) when \(\xi_v=|\,\cdot\,|_{F_v}\),
  and as the whole degenerate principal series in the other two cases.
  The decompositions and irreducibility assertions follow from \cite[Proposition~5.4]{Yamana2020HilbertHermitian} and \cite[Theorem~1.2]{KudlaSweet1997Degenerate},
  since \(r\) is odd.
  In the second case,
  \cite[Proposition~5.4(3),(4)]{Yamana2020HilbertHermitian} shows that the space of degenerate Whittaker functionals attached to \(B\) is nonzero for every nondegenerate \(B\),
  and its dimension is one by~\cite{Karel1979Functional}.
  In the third case,
  the characterization of the two constituents by the nonvanishing of their spaces of degenerate Whittaker functionals is given in \cite[Proposition~5.4(2),(4)]{Yamana2020HilbertHermitian}.
  By~\cite{Karel1979Functional},
  these spaces are again one-dimensional whenever they are nonzero.
\end{proof}

We now establish compatibility with the standard intertwining operators and determine the image of each odd-degree constituent.
We retain the local data \(E_v/F_v\),
\(S_v\),
\(\psi_v\),
and \(\chi_{-,v}\) fixed above.
For a quasi-character \(\zeta_v:E_v^\times\rightarrow\bbC^\times\),
write \(\Re(\zeta_v)=s\) if \(\zeta_v|\cdot|_{E_v}^{-s}\) is unitary.
Let
\[
  \beta_{S_v,\zeta_v}^{\chi_{-,v}}:
  I_{m+1,v}(\zeta_v)\otimes\overline{\omega_{S_v}}
  \longrightarrow I_{m,v}(\zeta_v\chi_{-,v}^{-1})
\]
be the \(G_m(F_v)\)-intertwining map in~\cite[Lemma~10.1]{Yamana2020HilbertHermitian},
with the normalization in~\eqref{eq:local-FJ-integral}.
Thus \(\beta_{S,v}^{\chi_-}=\beta_{S_v,\zeta_{m+1,v}}^{\chi_{-,v}}\).

Fix on \(\Her_r(F_v)\) the self-dual Haar measure for the pairing \((X,Z)\longmapsto\psi_v(\tr(XZ))\).
Put \(\zeta_v^\vee={}^c\zeta_v^{-1}\).
For \(\Re(\zeta_v)>r/2\),
the normalized standard intertwining operator \(M_{r,v}(\zeta_v):I_{r,v}(\zeta_v)\rightarrow I_{r,v}(\zeta_v^\vee)\) is given by
\[
  [M_{r,v}(\zeta_v)f](g)
  =
  \prod_{j=1}^{r}
  L\left(1-j, \left.\zeta_v\right|_{F_v^\times}\chi_v^{\,r+j}\right)^{-1}
  \int_{\Her_r(F_v)}f(J_rn_r(Z)g)\,dZ.
\]
The integral converges absolutely in this range,
and the right-hand side admits an entire continuation in \(\zeta_v\) by~\cite[\S~5]{Yamana2020HilbertHermitian};
see also~\cite[Proposition~3.2 and Theorem~1.3(5)]{KudlaSweet1997Degenerate}.

\begin{lem}
  \label{lem:FJ-intertwining-compatibility}
  There is a meromorphic scalar-valued function \(\gamma_{S,v}(\zeta_v)\) such that
  \[
    \beta_{S_v,\zeta_v^\vee}^{\chi_{-,v}}\circ
    \left(M_{m+1,v}(\zeta_v)\otimes\id_{\overline{\omega_{S_v}}}\right)
    =\gamma_{S,v}(\zeta_v)M_{m,v}(\zeta_v\chi_{-,v}^{-1})\circ\beta_{S_v,\zeta_v}^{\chi_{-,v}}.
  \]
  If \(\left.\zeta_v^\vee\right|_{F_v^\times}=|\,\cdot\,|_{F_v}\),
  then \(\gamma_{S,v}(\zeta_v)\in\bbC^\times\).
\end{lem}

\begin{proof}
  For \(r\in\{m,m+1\}\) and a quasi-character \(\xi_v:E_v^\times\rightarrow\bbC^\times\),
  let \(c_{r,v}(\xi_v)\) be the meromorphic scalar-valued function characterized by
  \[
    w_{B,v}^{\xi_v^\vee}\circ M_{r,v}(\xi_v)
    =c_{r,v}(\xi_v)\xi_v\left((\det B)^{-1}\right)
    \epsilon_{r,v}(B)^{r-1}w_{B,v}^{\xi_v}.
  \]
  This is~\cite[(5.3)]{Yamana2020HilbertHermitian};
  see also~\cite[Proposition~3.1, equations~(3.5) and~(3.9), and \S~7]{KudlaSweet1997Degenerate}.
  Put
  \begin{align*}
    b_{r,v}(\xi_v)
                   & =\prod_{j=1}^{r}
    L\left(j,
    \left.\xi_v\right|_{F_v^\times}\chi_v^{\,r+j}
    \right),                                            \\
    e_{r,v}(\xi_v) & =c_{r,v}(\xi_v)b_{r,v}(\xi_v^\vee)
  \end{align*}
  By~\cite[\S~5]{Yamana2020HilbertHermitian},
  \(e_{r,v}(\xi_v)\) is entire and nowhere vanishing.

  \cite[Lemma~10.1]{Yamana2020HilbertHermitian} gives \(C_{S,v}\in\bbC^\times\) such that
  \begin{align}
     & w_{T,v}^{\xi_v\chi_{-,v}^{-1}}
    \left(\beta_{S_v,\xi_v}^{\chi_{-,v}}(h\otimes\Phi)\right) \notag \\
     & =
    C_{S,v}|\det T|_{F_v}^{-1/2}
    \int_{M_{1,m}(E_v)}
    \overline{\Phi(x)}\,
    w_{S_v\oplus T,v}^{\xi_v}
    \left(I_{m+1,v}(\xi_v)(u_{m+1,1}(x,0,0))h\right)dx
    \label{eq:FJ-Whittaker-identity-local}
  \end{align}
  for every nondegenerate \(T\in\Her_m(F_v)\).

  Fix \(f\in I_{m+1,v}(\zeta_v)\),
  \(\Phi\in\overline{\omega_{S_v}}\),
  and nondegenerate \(T\in\Her_m(F_v)\).
  For generic \(\zeta_v\),
  apply~\eqref{eq:FJ-Whittaker-identity-local} with \(\xi_v=\zeta_v^\vee\) and \(h=M_{m+1,v}(\zeta_v)f\).
  Using the functional equation above with \(B=S_v\oplus T\),
  we obtain
  \begin{align*}
     & w_{T,v}^{\zeta_v^\vee\chi_{-,v}^{-1}}
    \left(
    \beta_{S_v,\zeta_v^\vee}^{\chi_{-,v}}
    \left(
    M_{m+1,v}(\zeta_v)f\otimes\Phi
    \right)
    \right)                                  \\
     & =
    c_{m+1,v}(\zeta_v)
    \zeta_v(S_v^{-1})
    \zeta_v((\det T)^{-1})
    \epsilon_{m+1,v}(S_v\oplus T)^m
    w_{T,v}^{\zeta_v\chi_{-,v}^{-1}}
    \left(
    \beta_{S_v,\zeta_v}^{\chi_{-,v}}(f\otimes\Phi)
    \right).
  \end{align*}
  On the other hand,
  the functional equation in rank \(m\) gives
  \begin{align*}
     & w_{T,v}^{\zeta_v^\vee\chi_{-,v}^{-1}}
    \left(
    M_{m,v}(\zeta_v\chi_{-,v}^{-1})
    \beta_{S_v,\zeta_v}^{\chi_{-,v}}(f\otimes\Phi)
    \right)                                  \\
     & =
    c_{m,v}(\zeta_v\chi_{-,v}^{-1})
    (\zeta_v\chi_{-,v}^{-1})((\det T)^{-1})
    \epsilon_{m,v}(T)^{m-1}
    w_{T,v}^{\zeta_v\chi_{-,v}^{-1}}
    \left(\beta_{S_v,\zeta_v}^{\chi_{-,v}}(f\otimes\Phi)\right).
  \end{align*}
  Since \(m\) is even and \(\left.\chi_{-,v}\right|_{F_v^\times}=\chi_v\),
  the quotient of the two scalar factors is
  \[
    \gamma_{S,v}(\zeta_v)
    =\frac{c_{m+1,v}(\zeta_v)\zeta_v(S_v^{-1})}{c_{m,v}(\zeta_v\chi_{-,v}^{-1})\chi_v((-1)^{m(m-1)/2})},
  \]
  which is independent of \(T\).
  Let \(D_{\zeta_v}\) be the difference of the two intertwining maps in the asserted identity.
  For generic \(\zeta_v\) and any nondegenerate \(T\),
  the preceding calculation gives
  \[
    \operatorname{Im}D_{\zeta_v}
    \subseteq\ker w_{T,v}^{\zeta_v^\vee\chi_{-,v}^{-1}}
    \subsetneq I_{m,v}(\zeta_v^\vee\chi_{-,v}^{-1}).
  \]
  Since \(I_{m,v}(\zeta_v^\vee\chi_{-,v}^{-1})\) is irreducible and \(w_{T,v}^{\zeta_v^\vee\chi_{-,v}^{-1}}\neq0\) by~\cite[Proposition~5.4(1) and Lemma~5.3(1)]{Yamana2020HilbertHermitian},
  equivariance implies \(D_{\zeta_v}=0\).
  The identity follows for all \(\zeta_v\) by meromorphic continuation.

  Now suppose that \(\left.\zeta_v^\vee\right|_{F_v^\times}=|\,\cdot\,|_{F_v}\).
  Put \(\zeta_v'=\zeta_v\chi_{-,v}^{-1}\).
  Then
  \begin{align*}
    \gamma_{S,v}(\zeta_v)
     & =
    \frac{e_{m+1,v}(\zeta_v)}{e_{m,v}(\zeta_v')}
    \frac{\zeta_v(S_v^{-1})}{\chi_v((-1)^{m(m-1)/2})}
    \frac{b_{m,v}((\zeta_v')^\vee)}{b_{m+1,v}(\zeta_v^\vee)} \\
     & =
    \frac{e_{m+1,v}(\zeta_v)}{e_{m,v}(\zeta_v')}
    \frac{\zeta_v(S_v^{-1})}{\chi_v((-1)^{m(m-1)/2})}
    L\left(m+1,|\cdot|_{F_v}\right)^{-1}.
  \end{align*}
  Every factor in the last expression is finite and nonzero,
  so \(\gamma_{S,v}(\zeta_v)\in\bbC^\times\).
\end{proof}

For \(d\in\{m,m+1\}\),
we denote the distinguished irreducible subrepresentations and constituents of \(I_{d,v}(\zeta_v)\) by \(A_{d,v}^\delta(\zeta_v)\),
following~\cite[Proposition~5.4]{Yamana2020HilbertHermitian}.

\begin{thm}
  \label{thm:local-one-block-FJ}
  Let \(\zeta_{m+1,v}:E_v^\times\rightarrow\bbC^\times\) be a character,
  and put \(\zeta_{m,v}=\zeta_{m+1,v}\chi_{-,v}^{-1}\).
  Then the following assertions hold.
  \begin{enumerate}
    \item If \(-1/2\leq\Re(\zeta_{m+1,v})<1/2\) and both \(I_{m+1,v}(\zeta_{m+1,v})\) and \(I_{m,v}(\zeta_{m,v})\) are irreducible,
          then
          \[
            \beta_{S_v,\zeta_{m+1,v}}^{\chi_{-,v}}\left(I_{m+1,v}(\zeta_{m+1,v})\otimes\overline{\omega_{S_v}}\right)
            = I_{m,v}(\zeta_{m,v}).
          \]

    \item If \(\left.\zeta_{m+1,v}\right|_{F_v^\times}=|\,\cdot\,|_{F_v}\),
          then
          \[
            \beta_{S_v,\zeta_{m+1,v}}^{\chi_{-,v}}\left(
            A_{m+1,v}^{+}(\zeta_{m+1,v})\otimes\overline{\omega_{S_v}}
            \right)
            =A_{m,v}^{+}(\zeta_{m,v}).
          \]

    \item If \(\left.\zeta_{m+1,v}\right|_{F_v^\times}=\chi_v^{\,m+1}\),
          then,
          for every \(\delta\in\{\pm1\}\),
          \[
            \beta_{S_v,\zeta_{m+1,v}}^{\chi_{-,v}}\left(A_{m+1,v}^{\delta}(\zeta_{m+1,v})\otimes\overline{\omega_{S_v}} \right)
            =A_{m,v}^{\delta\chi_v(S_v)}(\zeta_{m,v}).
          \]
  \end{enumerate}
\end{thm}

\begin{proof}
  Case~(1) follows from~\cite[Corollary~10.2]{Yamana2020HilbertHermitian} with \(i=1\).

  For case~(2),
  note that \(\zeta_{m,v}^\vee=\zeta_{m+1,v}^\vee\chi_{-,v}^{-1}\).
  By~\cite[Proposition~5.4(3) and Corollary~5.5(3)]{Yamana2020HilbertHermitian},
  \begin{align*}
    \Im M_{m+1,v}(\zeta_{m+1,v}^\vee) & =A_{m+1,v}^+(\zeta_{m+1,v}), \\
    \Im M_{m,v}(\zeta_{m,v}^\vee)     & =A_{m,v}^+(\zeta_{m,v}).
  \end{align*}
  Since \(\Re(\zeta_{m+1,v}^\vee)=-1/2\),
  the map \(\beta_{S_v,\zeta_{m+1,v}^\vee}^{\chi_{-,v}}\) is surjective by~\cite[Corollary~10.2]{Yamana2020HilbertHermitian}.
  Moreover,
  Lemma~\ref{lem:FJ-intertwining-compatibility} gives
  \[
    \beta_{S_v,\zeta_{m+1,v}}^{\chi_{-,v}}\circ
    \left(M_{m+1,v}(\zeta_{m+1,v}^\vee)\otimes\id_{\overline{\omega_{S_v}}}\right)
    =\gamma_{S,v}(\zeta_{m+1,v}^\vee)
    M_{m,v}(\zeta_{m,v}^\vee)\circ\beta_{S_v,\zeta_{m+1,v}^\vee}^{\chi_{-,v}},
  \]
  with \(\gamma_{S,v}(\zeta_{m+1,v}^\vee)\neq0\) because \(\zeta_{m+1,v}|_{F_v^\times}=|\cdot|_{F_v}\).
  Taking images proves case~(2).

  For case~(3),
  put \(\delta_{S_v}=\delta\chi_v(S_v)\),
  \(u(x)=u_{m+1,1}(x,0,0)\),
  and
  \[
    \calI=\beta_{S_v,\zeta_{m+1,v}}^{\chi_{-,v}}
    \left(A_{m+1,v}^{\delta}(\zeta_{m+1,v})\otimes\overline{\omega_{S_v}}\right).
  \]
  We first show that \(\calI\subseteq A_{m,v}^{\delta_{S_v}}(\zeta_{m,v})\).
  Since \(A_{m+1,v}^{\delta}(\zeta_{m+1,v})\) is a \(G_{m+1}(F_v)\)-subrepresentation of \(I_{m+1,v}(\zeta_{m+1,v})\),
  we have
  \[
    I_{m+1,v}(\zeta_{m+1,v})(u(x))h
    \in A_{m+1,v}^{\delta}(\zeta_{m+1,v})
  \]
  for all \(h\in A_{m+1,v}^{\delta}(\zeta_{m+1,v})\) and \(x\in M_{1,m}(E_v)\).
  By~\cite[Proposition~5.4(2),(4)]{Yamana2020HilbertHermitian},
  the integrand in~\eqref{eq:FJ-Whittaker-identity-local} therefore vanishes unless \(\epsilon_{m+1,v}(S_v\oplus T)=\delta\),
  which is equivalent to \(\epsilon_{m,v}(T)=\delta_{S_v}\).
  Thus \(w_{T,v}^{\zeta_{m,v}}|_{\calI}=0\) if \(\epsilon_{m,v}(T)=-\delta_{S_v}\).
  In the decomposition
  \[
    I_{m,v}(\zeta_{m,v})=A_{m,v}^{\delta_{S_v}}(\zeta_{m,v})\oplus A_{m,v}^{-\delta_{S_v}}(\zeta_{m,v}),
  \]
  the projection of \(\calI\) onto the second summand must consequently be zero.
  Otherwise it would be surjective by irreducibility,
  contradicting the Whittaker vanishing above and~\cite[Proposition~5.4(2),(4)]{Yamana2020HilbertHermitian}.
  This proves the desired inclusion.

  It remains to show that \(\calI\neq0\).
  Choose a nondegenerate \(T\in\Her_m(F_v)\) with \(\epsilon_{m,v}(T)=\delta_{S_v}\).
  Then \(\epsilon_{m+1,v}(S_v\oplus T)=\delta\),
  so~\cite[Proposition~5.4(2),(4)]{Yamana2020HilbertHermitian} gives \(h\in A_{m+1,v}^{\delta}(\zeta_{m+1,v})\) with \(w_{S_v\oplus T,v}^{\zeta_{m+1,v}}(h)\neq0\).
  By local constancy,
  there is a compact open neighborhood \(X\) of \(0\) in \(M_{1,m}(E_v)\) such that
  \[
    w_{S_v\oplus T,v}^{\zeta_{m+1,v}}\left(I_{m+1,v}(\zeta_{m+1,v})(u(x))h\right)
    =w_{S_v\oplus T,v}^{\zeta_{m+1,v}}(h)
  \]
  for all \(x\in X\).
  Taking \(\Phi=1_X\) in~\eqref{eq:FJ-Whittaker-identity-local} yields
  \[
    w_{T,v}^{\zeta_{m,v}}\left(\beta_{S_v,\zeta_{m+1,v}}^{\chi_{-,v}}(h\otimes1_X)\right)
    =C_{S,v}|\det T|_{F_v}^{-1/2}\vol(X)w_{S_v\oplus T,v}^{\zeta_{m+1,v}}(h)\neq0.
  \]
  Hence \(\calI\) is a nonzero subrepresentation of the irreducible representation \(A_{m,v}^{\delta_{S_v}}(\zeta_{m,v})\),
  and therefore \(\calI=A_{m,v}^{\delta_{S_v}}(\zeta_{m,v})\).
\end{proof}

\subsection{Dihedral supercuspidal representations}

Fix a nonsplit finite place \(v\) such that \(\pi_{h,v}\) is supercuspidal and \(\pi_{h,v}\simeq\pi_{h,v}\otimes\chi_v\) throughout this subsection.

By Clifford theory,
choose a character \(\theta_v\neq\theta_v^c\) of \(W_{E_v}\) such that
\[
  \phi_{\pi_{h,v}}
  \simeq
  \Ind_{W_{E_v}}^{W_{F_v}}\theta_v.
\]
If \(h\) has CM by \(E\),
choose \(\theta_v\) to be the localization of \(\theta\).

Via local class field theory,
we use the same notation for characters of \(W_{E_v}\) and of \(E_v^\times\).
Put \(\gamma_v=\theta_v(\theta_v^c)^{-1}\).
Then \(\gamma_v\neq\one\),
\(\gamma_v^c=\gamma_v^{-1}\),
and \(\gamma_v|_{F_v^\times}=\one\).
Put \(E_v^1=\ker(N_{E_v/F_v}:E_v^\times\rightarrow F_v^\times)\).
Hilbert's Theorem~90 gives a unique character \(\gamma_v^\natural:E_v^1\rightarrow\bbC^\times\) such that \(\gamma_v^\natural(x/x^c)=\gamma_v(x)\) for \(x\in E_v^\times\).

For \(n\in\{m,m+1\}\),
put
\[
  \lambda_{n,v}=((\vartheta_{n,v})^c)^{-1}\theta_v^c,
  \ \text{and}\
  \lambda_{n,v}|_{F_v^\times}=\chi_v^n.
\]
By the definition of \(\psi_n(h)\),
its localization is
\begin{equation}
  \label{eq:dihedral-local-parameter-unified}
  \psi_{n,v}(h)=(\gamma_v\lambda_{n,v})[n]\boxplus\lambda_{n,v}[n].
\end{equation}

For an \(n\)-dimensional nondegenerate Hermitian space \(V\) over \(E_v\) with Gram matrix \(A\),
put
\[
  \epsilon_{n,v}(V)=\chi_v\!\left((-1)^{n(n-1)/2}\det A\right).
\]
This sign depends only on the isometry class of \(V\) and distinguishes the two such classes.
For \(\delta\in\{\pm1\}\),
fix a representative \(V_{n,v}^{\delta}\) such that \(\epsilon_{n,v}(V_{n,v}^{\delta})=\delta\).

Let \(W_{2n,v}\) be the split skew-Hermitian space over \(E_v\) with Gram matrix \(J_n\),
so that \(\U(W_{2n,v})=G_n(F_v)\).
For characters \(\chi_V,\chi_W:E_v^\times\longrightarrow\bbC^\times\) satisfying
\begin{align*}
  \chi_V|_{F_v^\times} & =\chi_v^{\dim V}, \\
  \chi_W|_{F_v^\times} & =\chi_v^{\dim W},
\end{align*}
let \(\Omega(V,W;\chi_V,\chi_W)\) denote the Weil representation of \(\U(V)\times\U(W)\) associated with \(\psi_v\),
with the splitting convention of~\cite[\S~4.1]{GanIchino2016GrossPrasad}.
For an irreducible smooth representation \(\sigma\) of \(\U(V)\),
define its big theta lift to \(\U(W)\) by
\[
  \Theta(V,W;\chi_V,\chi_W)(\sigma):=(\Omega(V,W;\chi_V,\chi_W)\otimes\sigma^\vee)_{\U(V)}.
\]
Chen--Gan~\cite[\S~2.1]{ChenGan2025TwistedGGP} use the opposite convention for the types of the two spaces:
their first space is skew-Hermitian and their second is Hermitian.
Multiplying the forms on our first and second spaces by a trace-zero scalar in \(E_v^\times\) and its inverse, respectively, converts our pair into their convention without changing either unitary group or the tensor-product symplectic form.
We make this identification when applying their results below.

For \(n\in\{m,m+1\}\) and \(\delta\in\{\pm1\}\),
define
\[
  \Theta_{n,v}^{\delta}=\Theta(V_{n,v}^{\delta},W_{2n,v};\lambda_{n,v},\one)(\gamma_v^\natural\circ\det).
\]

For nondegenerate \(T\in\Her_n(F_v)\) and a smooth representation \(\sigma\) of \(G_n(F_v)\),
we retain the notation \(\epsilon_{n,v}(T)\),
\(\psi_{T,v}\),
and \(\Wh_{T,v}(\sigma)\) introduced before Proposition~\ref{prop:odd-nonsupercuspidal-structure}.

We first distinguish the two theta lifts using their spaces of degenerate Whittaker functionals.

\begin{lem}
  \label{lem:local-dihedral-theta-whittaker}
  For every nondegenerate \(T\in\Her_n(F_v)\),
  \begin{equation}
    \label{eq:dihedral-theta-Whittaker}
    \dim\Wh_{T,v}(\Theta_{n,v}^{\delta})
    =
    \begin{cases}
      1 & \text{if }\epsilon_{n,v}(T)=\delta,    \\
      0 & \text{if }\epsilon_{n,v}(T)\neq\delta.
    \end{cases}
  \end{equation}
\end{lem}

\begin{proof}
  We first express the Whittaker space in terms of tuples with Gram matrix \(T\).
  For \(x=(x_1,\ldots,x_n)\in(V_{n,v}^{\delta})^n\),
  write
  \[
    Q(x)
    =
    ((x_i,x_j)_{V_{n,v}^{\delta}})_{1\leq i,j\leq n}
    \in\Her_n(F_v).
  \]
  The unipotent radical of the Siegel parabolic acts in the Weil representation by
  \[
    (\omega(n_n(Z))\Phi)(x)
    =
    \psi_v\!\left(\tr(Q(x)Z)\right)\Phi(x).
  \]
  Put
  \[
    \Omega_T
    =
    \{x\in(V_{n,v}^{\delta})^n\mid Q(x)=T\}.
  \]
  Restriction to \(\Omega_T\) induces an isomorphism of smooth \(\U(V_{n,v}^{\delta})\)-representations
  \begin{align*}
    \calS((V_{n,v}^{\delta})^n)_{N_n(F_v),\psi_{T,v}}
           & \xrightarrow{\ \sim\ }
    C_c^\infty(\Omega_T),                 \\
    [\Phi] & \longmapsto\Phi|_{\Omega_T}.
  \end{align*}

  The actions of \(\U(V_{n,v}^{\delta})\) and \(N_n(F_v)\) commute.
  Thus the definition of the big theta lift and the universal property of coinvariants give
  \begin{align*}
    \Wh_{T,v}(\Theta_{n,v}^{\delta})
     & =
    \Hom_{N_n(F_v)}
    (\Theta_{n,v}^{\delta},\psi_{T,v}) \\
     & \simeq
    \Hom_{\U(V_{n,v}^{\delta})\times N_n(F_v)}
    \left(
    \Omega(V_{n,v}^{\delta},W_{2n,v};\lambda_{n,v},\one),
    (\gamma_v^\natural\circ\det)\boxtimes\psi_{T,v}
    \right)                            \\
     & \simeq
    \Hom_{\U(V_{n,v}^{\delta})}
    \left(
    C_c^\infty(\Omega_T),
    \gamma_v^\natural\circ\det
    \right).
  \end{align*}
  Since \(T\) is nondegenerate and \(\dim V_{n,v}^{\delta}=n\),
  a tuple in \(\Omega_T\) is a basis with Gram matrix \(T\).
  Hence \(\Omega_T\) is nonempty if and only if \(\epsilon_{n,v}(T)=\delta\).
  If \(\epsilon_{n,v}(T)\neq\delta\),
  then \(C_c^\infty(\Omega_T)=0\),
  so the displayed identification gives \(\Wh_{T,v}(\Theta_{n,v}^{\delta})=0\).

  Now suppose that \(\epsilon_{n,v}(T)=\delta\).
  Any two bases with Gram matrix \(T\) are related by a unique element of \(\U(V_{n,v}^{\delta})\).
  Thus this group acts simply transitively on \(\Omega_T\),
  and choosing one such basis identifies \(C_c^\infty(\Omega_T)\) with its regular representation.
  Compact Frobenius reciprocity gives
  \[
    \Hom_{\U(V_{n,v}^{\delta})}
    \left(
    C_c^\infty(\Omega_T),
    \gamma_v^\natural\circ\det
    \right)
    \simeq
    \Hom_{\{1\}}(\one,\one)
    \simeq \bbC.
  \]

  This proves~\eqref{eq:dihedral-theta-Whittaker}.
\end{proof}

\begin{prop}
  \label{prop:dihedral-local-constituents}
  For \(n\in\{m,m+1\}\),
  \begin{equation}
    \label{eq:theta-target-packet}
    \Pi_{\psi_{n,v}(h)}
    =
    \{\Theta_{n,v}^+,\Theta_{n,v}^-\}.
  \end{equation}
  In particular,
  \[
    \left.\Pi_{m+1,v}(h)\right|_{G_{m+1}(F_v)}
    \simeq
    \Theta_{m+1,v}^+
    \oplus
    \Theta_{m+1,v}^-.
  \]
\end{prop}
\begin{proof}
  Fix \(n\in\{m,m+1\}\).
  For each \(\delta\in\{\pm1\}\),
  choose a Gram matrix \(T\) of \(V_{n,v}^{\delta}\).
  Lemma~\ref{lem:local-dihedral-theta-whittaker} gives
  \[
    \dim\Wh_{T,v}(\Theta_{n,v}^{\delta})=1,
    \qquad
    \dim\Wh_{T,v}(\Theta_{n,v}^{-\delta})=0.
  \]
  Thus both theta lifts are nonzero and they are not isomorphic.
  To prove irreducibility,
  note that the split space \(W_{2n,v}\) has Witt index \(n=\dim V_{n,v}^{\delta}\).
  Hence the dual pair \(\U(V_{n,v}^{\delta})\times G_n(F_v)\) is in the stable range with \(\U(V_{n,v}^{\delta})\) the smaller member.
  Moreover,
  the character \(\gamma_v^\natural\circ\det\) is unitary because \(E_v^1\) is compact.
  Corollary~7.3 of~\cite{ChenZou2024BigTheta} therefore shows that \(\Theta_{n,v}^{\delta}\) is irreducible.

  The trivial representation of \(\U(V_{n,v}^{\delta})\) belongs to the Arthur packet attached to \(\one[n]\).
  By the description of the Arthur packets for \(\one[n]\) and \(\gamma_v[n]\) in terms of Aubert duality and the compatibility of the local Langlands correspondence with determinant twists \cite[Theorem~2.5.1(9)]{ChenZou2021LLC},
  \[
    \Pi_{\gamma_v[n]}(\U(V_{n,v}^{\delta}))
    =
    \{\gamma_v^\natural\circ\det\};
  \]
  cf.~\cite[Remark~2.2(2) and Lemma~2.4]{ChenGan2025TwistedGGP}.

  We now apply~\cite[Theorem~2.1(1)]{ChenGan2025TwistedGGP} to this packet.
  In the notation of Chen--Gan~\cite[\S~2.2]{ChenGan2025TwistedGGP},
  take the \(A\)-parameter \(\gamma_v[n]\),
  with splitting characters \(\chi_W=\one\) and \(\chi_V=\lambda_{n,v}\).
  Since \(\gamma_v\neq\one\),
  the condition in \cite[Theorem~2.1(1)]{ChenGan2025TwistedGGP} is vacuous,
  and the \(A\)-parameter prescribed there on \(G_n(F_v)\) is
  \[
    (\gamma_v\lambda_{n,v})[n]
    \boxplus
    \lambda_{n,v}[n]
    =
    \psi_{n,v}(h).
  \]
  Here the \(A\)-packets for \(\gamma_v[n]\) and \(\psi_{n,v}(h)\) are Zelevinsky--Aubert duals of tempered \(L\)-packets.
  As explained in~\cite[Remark~2.2(2)]{ChenGan2025TwistedGGP},
  the properties of these packets needed in Theorem~2.1(1) follow from the established local Langlands correspondence and Zelevinsky--Aubert duality.
  Since \(\Theta_{n,v}^{\delta}\) is nonzero and irreducible,
  \cite[Theorem~2.1(1)]{ChenGan2025TwistedGGP} gives
  \[
    \Theta_{n,v}^{\delta}
    \in
    \Pi_{\psi_{n,v}(h)}.
  \]

  It remains to count the members of the packet.
  Exchange the Arthur and Weil--Deligne \(\SL_2(\bbC)\)-factors in \(\psi_{n,v}(h)\) to obtain the tempered \(L\)-parameter
  \[
    \phi_{n,v}
    =
    (\gamma_v\lambda_{n,v})\otimes \Sym^{n-1}(\bbC^2)
    \boxplus
    \lambda_{n,v}\otimes \Sym^{n-1}(\bbC^2)
  \]
  where \(\Sym^{n-1}(\bbC^2)\) occurs on the Weil--Deligne factor.
  Since \(\gamma_v\neq\one\),
  \(\phi_{n,v}\) has two distinct irreducible summands and component group \((\bbZ/2\bbZ)^2\).
  By the parametrization in \cite[Theorem~2.5.1]{ChenZou2021LLC},
  the members of the Vogan packet are parametrized by characters of this component group,
  and exactly two of them correspond to the quasi-split group \(G_n(F_v)\).
  Thus the tempered packet \(\Pi_{\phi_{n,v}}\) on \(G_n(F_v)\) has two members.

  The parameter \(\psi_{n,v}(h)\) is of good parity,
  multiplicity-free,
  and trivial on the Weil--Deligne \(\SL_2(\bbC)\)-factor.
  Hence the Zelevinsky--Aubert duality described in \cite[Remark~2.2(2) and Lemma~2.4]{ChenGan2025TwistedGGP} identifies \(\Pi_{\psi_{n,v}(h)}\) with the Aubert duals of the members of \(\Pi_{\phi_{n,v}}\).
  Since Aubert duality is an involution on irreducible representations,
  \(\Pi_{\psi_{n,v}(h)}\) also has exactly two members,
  so the two distinct members already found exhaust it:
  \[
    \Pi_{\psi_{n,v}(h)}
    =
    \{\Theta_{n,v}^+,\Theta_{n,v}^-\}.
  \]

  Finally,
  take \(n=m+1\).
  Lemma~\ref{lem:odd-supercuspidal-properties} expresses \(\left.\Pi_{m+1,v}(h)\right|_{G_{m+1}(F_v)}\) as the direct sum of two nonisomorphic irreducible members of this packet.
  The packet equality just proved identifies these members with \(\Theta_{m+1,v}^+\) and \(\Theta_{m+1,v}^-\),
  which gives the final assertion.
\end{proof}
With the constituents identified,
we now compute their Fourier--Jacobi modules.

\begin{thm}
  \label{thm:local-dihedral-FJ}
  For every \(\delta\in\{\pm1\}\) and \(S_v\in F_v^\times\),
  \[
    \FJ_{S_v}(\Theta_{m+1,v}^{\delta})
    \simeq
    \Theta_{m,v}^{\delta\chi_v(S_v)}.
  \]
\end{thm}

\begin{proof}
  Put \(V=V_{m+1,v}^{\delta}\),
  \(H=\U(V)\),
  \(U=U_{m+1,1}(F_v)\),
  \(Z=Z_{m+1,1}(F_v)\),
  and \(\eta=\gamma_v^\natural\circ\det\).
  Since \(m+1\geq3\) is odd,
  \(V\) is isotropic and hence represents every element of \(F_v^\times\).

  Write
  \[
    W_{2m+2,v}
    =
    E_ve\oplus W_{2m,v}\oplus E_vf
  \]
  with \(E_ve\oplus E_vf\) hyperbolic.
  In the mixed Schr\"odinger model of
  \cite[\S~7.4]{GanIchino2016GrossPrasad},
  \[
    \Omega(V,W_{2m+2,v};\lambda_{m+1,v},\one)
    \simeq
    \calS(V)\otimes
    \Omega(V,W_{2m,v};\lambda_{m+1,v},\one)
  \]
  as an \(H\times G_m(F_v)\)-representation.

  Set
  \[
    \frakO
    =
    \Omega(V,W_{2m+2,v};\lambda_{m+1,v},\one)
    \otimes\overline{\omega_{S_v}}.
  \]
  For \(z(t)=u_{m+1,1}(0,0,t)\in Z\),
  the action on \(\frakO\) in this model is
  \[
    (z(t)\Phi)(x)
    =
    \psi_v\!\left(t\bigl((x,x)_V-S_v\bigr)\right)\Phi(x).
  \]
  Taking \(Z\)-coinvariants therefore amounts to restricting functions to
  \[
    X_{S_v}(V)
    =
    \{x\in V\mid (x,x)_V=S_v\}.
  \]
  Choose \(x\in X_{S_v}(V)\) and put \(V_x=(E_vx)^\perp\) and \(H_x=\U(V_x)\).
  By Witt's theorem,
  \(X_{S_v}(V)\) is a single \(H\)-orbit with stabilizer \(H_x\).
  Thus we have an isomorphism of \(H\times G_m(F_v)\)-representations
  \[
    \frakO_Z
    \simeq
    \mathrm{c\text{-}Ind}_{H_x}^{H}
    \left(
    \Omega(V,W_{2m,v};\lambda_{m+1,v},\one)
    \otimes\overline{\omega_{S_v}}
    \right).
  \]

  Since \(V=E_vx\perp V_x\) and \(\lambda_{m+1,v}=\chi_{-,v}\lambda_{m,v}\),
  the decomposition property of the Weil representation gives
  \[
    \Omega(V,W_{2m,v};\lambda_{m+1,v},\one)
    \simeq
    \omega_{S_v}\otimes
    \Omega(V_x,W_{2m,v};\lambda_{m,v},\one)
  \]
  as an \(H_x\times G_m(F_v)\)-representation.
  Under the preceding compact-induction realization of \(\frakO_Z\),
  the inducing representation attached to \(x\) is therefore
  \[
    \omega_{S_v}\otimes\overline{\omega_{S_v}}
    \otimes
    \Omega(V_x,W_{2m,v};\lambda_{m,v},\one).
  \]
  The induced \(U/Z\)-action is diagonal on the first two factors and trivial on \(\Omega(V_x,W_{2m,v};\lambda_{m,v},\one)\).
  Since the central characters of \(\omega_{S_v}\) and \(\overline{\omega_{S_v}}\) cancel,
  this indeed defines an action of \(U/Z\).
  By the Stone--von Neumann theorem and Schur's lemma,
  the natural pairing induces an isomorphism
  \[
    \left(
    \omega_{S_v}\otimes\overline{\omega_{S_v}}
    \right)_{U/Z}
    \simeq
    \bbC.
  \]
  Consequently,
  \[
    \frakO_U
    \simeq
    \mathrm{c\text{-}Ind}_{H_x}^{H}
    \Omega(V_x,W_{2m,v};\lambda_{m,v},\one).
  \]

  Taking \(H\)-coinvariants and using compact induction,
  we obtain
  \begin{align*}
    \FJ_{S_v}(\Theta_{m+1,v}^{\delta})
     & =
    \left(
    \Theta_{m+1,v}^{\delta}
    \otimes\overline{\omega_{S_v}}
    \right)_U
    \\
     & \simeq
    \left(
    \Omega(V_x,W_{2m,v};\lambda_{m,v},\one)
    \otimes(\eta|_{H_x})^\vee
    \right)_{H_x}
    \\
     & =
    \Theta(V_x,W_{2m,v};\lambda_{m,v},\one)
    (\gamma_v^\natural\circ\det_{V_x}).
  \end{align*}
  Here \(\eta|_{H_x}=\gamma_v^\natural\circ\det_{V_x}\) because \(H_x\) fixes \(x\).

  It remains to identify \(V_x\).
  If \(A_x\) is a Gram matrix of \(V_x\),
  then \(\diag(S_v,A_x)\) is a Gram matrix of \(V\).
  Hence
  \begin{align*}
    \delta
     & =
    \chi_v\!\left(
    (-1)^{m(m+1)/2}S_v\det A_x
    \right)
    \\
     & =
    \chi_v(S_v)\epsilon_{m,v}(V_x),
  \end{align*}
  since \(m\) is even.
  Thus \(V_x\simeq V_{m,v}^{\delta\chi_v(S_v)}\),
  so the theta lift above is \(\Theta_{m,v}^{\delta\chi_v(S_v)}\).
\end{proof}
\begin{rem}
  The same mixed-model argument applies to the reducible degenerate principal-series case \(\left.\zeta_{m+1,v}\right|_{F_v^\times}=\chi_v^{\,m+1}\),
  giving
  \[
    \FJ_{S_v}
    \left(
      A_{m+1,v}^{\delta}(\zeta_{m+1,v})
    \right)
    \simeq
    A_{m,v}^{\delta\chi_v(S_v)}(\zeta_{m,v}).
  \]
\end{rem}

\subsection{Singleton local packets}

Suppose that either \(v\) is split in \(E/F\),
or that \(E_v/F_v\) is a field and \(\pi_{h,v}\) is supercuspidal with
\[
  \pi_{h,v}\not\simeq\pi_{h,v}\otimes\chi_v.
\]
In the split case,
\cite[Appendix~B]{Yamana2020HilbertHermitian} shows that both the odd- and even-degree local packets are singletons.
In the nonsplit supercuspidal case,
the odd-degree assertion is Lemma~\ref{lem:odd-supercuspidal-properties}.

In even degree,
let \(\phi_{m,v}\) be the tempered \(L\)-parameter obtained from \(\psi_{m,v}(h)\) by exchanging the Arthur and Weil--Deligne \(\SL_2(\bbC)\)-factors.
It has a single irreducible summand,
so its local component group is isomorphic to \(\bbZ/2\bbZ\),
with nontrivial distinguished central element.
By~\cite[Theorem~2.5.1]{ChenZou2021LLC},
there is therefore exactly one member of the corresponding \(L\)-packet on the quasi-split group \(G_m(F_v)\).
By the Zelevinsky--Aubert duality described in \cite[Remark~2.2(2) and Lemma~2.4]{ChenGan2025TwistedGGP},
the same is true of \(\Pi_{\psi_{m,v}(h)}\).

By Proposition~\ref{prop:even-nonvanishing},
choose an irreducible automorphic constituent \(\sigma\) of \(\calV_{m,S}(h)\),
arising from an irreducible constituent \(\rho\) of the odd-degree lift.
Localizing the global Fourier--Jacobi map at \(v\) gives
\[
  \Hom_{G_m(F_v)}
  \left(
  \FJ_{S_v}(\rho_v),
  \sigma_v
  \right)
  \neq0.
\]
By Corollary~\ref{cor:global-packet-identification},
\(\sigma_v\) is the unique member of \(\Pi_{\psi_{m,v}(h)}\).
Hence the nonvanishing of the Hom space above gives the asserted descent from the unique odd-degree local constituent.

\subsection{Local constituents and descent}

Let \(\Sigma_{\spl}\) denote the set of split finite places.

Suppose that \(v\) is nonsplit and \(\pi_{h,v}\) is not supercuspidal.
Recall that
\[
  \xi_v
  =
  \left.\zeta_{m+1,v}\right|_{F_v^\times}
  =
  \mu_v^2\omega_{h,v}^{-1}.
\]
According to the value of \(\xi_v\),
define
\begin{align*}
  \Sigma_{\pr,\irr}
   & =\{v\mid\xi_v\notin\{\chi_v,|\cdot|_{F_v}\}\}, \\
  \Sigma_{\st}
   & =\{v\mid\xi_v=|\cdot|_{F_v}\},                 \\
  \Sigma_{\pr,2}
   & =\{v\mid\xi_v=\chi_v\},
\end{align*}
where in each case \(v\) is understood to be nonsplit and \(\pi_{h,v}\) nonsupercuspidal.

For the nonsplit supercuspidal places,
put
\begin{align*}
  \Sigma_{\sc}
   & =\{v\mid\pi_{h,v}\not\simeq\pi_{h,v}\otimes\chi_v\}, \\
  \Sigma_{\dh}
   & =\{v\mid\pi_{h,v}\simeq\pi_{h,v}\otimes\chi_v\}.
\end{align*}
Together with \(\Sigma_{\spl}\),
these five sets partition the finite places.

Finally,
put
\[
  \Sigma_2=\Sigma_{\pr,2}\sqcup\Sigma_{\dh}.
\]
These are precisely the places at which \(\Pi_{\psi_{n,v}(h)}\) has two members for \(n\in\{m,m+1\}\).

\begin{rem}\label{rem:global-CM-local-types}
  If \(h\) has CM by \(E\),
  then \(\Sigma_{\pr,\irr}=\Sigma_{\st}=\Sigma_{\sc}=\varnothing\).
  Thus every nonsplit finite place lies in \(\Sigma_{\pr,2}\sqcup\Sigma_{\dh}\),
  and the odd-degree restriction has two constituents there.
\end{rem}

For \(n\in\{m,m+1\}\),
we define the local constituents by
\[
  \rho_{n,v}=
  \begin{cases}
    \text{the unique member of }\Pi_{\psi_{n,v}(h)}
                             & (v\in\Sigma_{\spl}\sqcup\Sigma_{\sc}), \\
    I_{n,v}(\zeta_{n,v})     & (v\in\Sigma_{\pr,\irr}),               \\
    A_{n,v}^{+}(\zeta_{n,v}) & (v\in\Sigma_{\st}),
  \end{cases}
\]
and,
for \(\delta\in\{\pm1\}\),
\[
  \rho_{n,v}^{\delta}=
  \begin{cases}
    A_{n,v}^{\delta}(\zeta_{n,v}) & (v\in\Sigma_{\pr,2}), \\
    \Theta_{n,v}^{\delta}         & (v\in\Sigma_{\dh}).
  \end{cases}
\]

\begin{thm}
  \label{thm:complete-local-selection}
  Let \(v\in\bfh\).

  If \(v\in\Sigma_2\),
  then for every \(\delta\in\{\pm1\}\),
  \[
    \Hom_{G_m(F_v)}\!\left(
    \FJ_{S_v}(\rho_{m+1,v}^{\delta}),
    \rho_{m,v}^{\delta\chi_v(S_v)}
    \right)\neq0.
  \]
  If \(v\notin\Sigma_2\),
  then
  \[
    \Hom_{G_m(F_v)}\!\left(
    \FJ_{S_v}(\rho_{m+1,v}),
    \rho_{m,v}
    \right)\neq0.
  \]
\end{thm}

\begin{proof}
  The assertions for principal series at nonsplit places are exactly those of Theorem~\ref{thm:local-one-block-FJ},
  using the factorization of \(\beta_{S,v}^{\chi_-}\) through the local Fourier--Jacobi module described after Definition~\ref{def:local-FJ-module}.
  The dihedral assertion follows from Proposition~\ref{prop:dihedral-local-constituents} and Theorem~\ref{thm:local-dihedral-FJ}.
  The split and non-dihedral supercuspidal cases follow from the results above,
  since the relevant local packets are singletons.
\end{proof}

\section{Global decomposition}
\label{sec:finite-part}

Fix \(S\in F_{>0}\) throughout this section.

At every good place \(v\in\Sigma_2\) at which \(\chi_{-,v}\) is unramified and \(S\in\calO_{F_v}^\times\),
we have \(v\in\Sigma_{\pr,2}\),
and the spherical constituent is the \(+\)-constituent in both degrees.
Indeed,
for an integral unimodular \(B_v\),
\(\epsilon_{r,v}(B_v)=+1\),
and~\eqref{eq:spherical-Whittaker-normalization} identifies the spherical constituents with \(A_{m+1,v}^{+}(\zeta_v^+)\) and \(A_{m,v}^{+}(\zeta_v^-)\),
respectively.
No place in \(\Sigma_{\dh}\) is good.

Let
\[
  \frakS
  =
  \left\{
  (\delta_v)_{v\in\Sigma_2}
  \ \middle|\
  \delta_v\in\{\pm1\},\
  \delta_v=+1\text{ for almost all }v
  \right\}.
\]
For \(\delta\in\frakS\),
put
\[
  \rho_{m+1}(\delta)
  =
  \left({\bigotimes_{v\in\Sigma_2}}'\rho_{m+1,v}^{\delta_v}\right)
  \otimes
  \left({\bigotimes_{v\notin\Sigma_2}}'\rho_{m+1,v}\right).
\]
Put \(\delta^S=(\delta_v\chi_v(S_v))_{v\in\Sigma_2}\) and
\[
  \rho_m(\delta^S)
  =
  \left({\bigotimes_{v\in\Sigma_2}}'\rho_{m,v}^{\delta_v\chi_v(S_v)}\right)
  \otimes
  \left({\bigotimes_{v\notin\Sigma_2}}'\rho_{m,v}\right).
\]

Taking the restricted tensor product of the local decompositions gives the following algebraic direct sum decomposition:
\begin{equation}\label{eq:odd-global-restriction}
  \left.\Pi_{m+1}(h)\right|_{G_{m+1}(\bbA_\bfh)}
  \simeq
  \bigoplus_{\delta\in\frakS}\rho_{m+1}(\delta).
\end{equation}

Put \(\Sigma_1=\Sigma_{\pr,\irr}\sqcup\Sigma_{\st}\sqcup\Sigma_{\sc}\) and \(d=[F:\bbQ]\).

\begin{lem}
  \label{lem:global-sign-condition}
  Let \(\delta\in\frakS\).
  There exists \(T\in\Her_m(F)_{>0}\) such that \(\epsilon_{m+1,v}(S\oplus T)=\delta_v\) for every \(v\in\Sigma_2\) if and only if either \(\Sigma_1\neq\varnothing\),
  or \(\Sigma_1=\varnothing\) and
  \[
    \prod_{v\in\Sigma_2}\delta_v=(-1)^{dm(m+1)/2}.
  \]
\end{lem}

\begin{proof}
  For \(T\in\Her_m(F)_{>0}\),
  set \(b_T=(-1)^{m(m+1)/2}S\det T\).
  Then \(\epsilon_{m+1,v}(S\oplus T)=\chi_v(b_T)\),
  so the required local conditions are equivalent to \(\chi_v(b_T)=\delta_v\) for every \(v\in\Sigma_2\).
  As \(S\) is totally positive,
  the sign of \(b_T\) at every real place is \((-1)^{m(m+1)/2}\).

  If \(\Sigma_1=\varnothing\),
  every finite place outside \(\Sigma_2\) is split in \(E\),
  and hence \(\chi_v\) is trivial there.
  The product formula applied to \(b_T\) therefore gives
  \[
    \prod_{v\in\Sigma_2}\delta_v=(-1)^{dm(m+1)/2},
  \]
  proving the necessity.

  Conversely,
  suppose that one of the two conditions in the statement holds,
  and prescribe local signs by
  \[
    s_v=
    \begin{cases}
      \delta_v        & (v\in\Sigma_2),     \\
      (-1)^{m(m+1)/2} & (v\mid\infty),      \\
      1               & (\text{otherwise}).
    \end{cases}
  \]
  If \(\Sigma_1\neq\varnothing\),
  choose \(v_0\in\Sigma_1\) and replace the value of \(s_{v_0}\) by
  \[
    s_{v_0}=(-1)^{dm(m+1)/2}\prod_{v\in\Sigma_2}\delta_v.
  \]
  Then \(\prod_v s_v=1\).
  If \(\Sigma_1=\varnothing\),
  the assumed product condition gives \(\prod_v s_v=1\).
  By the local--global exact sequence for norm classes \cite[Chapter~VI, Proposition~5.6 and Corollary~5.7]{Neukirch1999ANT},
  there is \(b\in F^\times\) such that \(\chi_v(b)=s_v\) at every place.
  Put \(a=b/S\) and \(T=\diag(1,\ldots,1,(-1)^{m(m+1)/2}a)\).
  Then \(T\) is positive definite and \(\epsilon_{m+1,v}(S\oplus T)=\chi_v(b)=\delta_v\) for every \(v\in\Sigma_2\).
\end{proof}

Put
\[
  \frakS_0
  =
  \left\{
    \delta\in\frakS
    \;\middle|\;
    \begin{array}{c}
      \text{there exists }T\in\Her_m(F)_{>0}\text{ such that}\\
      \epsilon_{m+1,v}(S\oplus T)=\delta_v\text{ for every }v\in\Sigma_2
    \end{array}
  \right\}.
\]
By Lemma~\ref{lem:global-sign-condition},
this subset is all of \(\frakS\) if \(\Sigma_1\neq\varnothing\);
if \(\Sigma_1=\varnothing\),
it consists precisely of the \(\delta\in\frakS\) satisfying
\[
  \prod_{v\in\Sigma_2}\delta_v=(-1)^{dm(m+1)/2}.
\]
Let \(\calD_{m,\infty}^{\hol}\) denote the holomorphic lowest-weight representation of \(G_m(F_\infty)\) of scalar weight \((\ell_m,\eta_m)\).

\begin{thm}
  \label{thm:exact-finite-decomposition}
  The irreducible automorphic representations occurring in \(\calV_{m,S}(h)\) are precisely \(\calD_{m,\infty}^{\hol}\otimes\rho_m(\delta^S)\) with \(\delta\in\frakS_0\).
  They are pairwise nonisomorphic and occur with multiplicity one.
  In particular,
  \[
    \Pi_{m,S}(h)\simeq\bigoplus_{\delta\in\frakS_0}\rho_m(\delta^S).
  \]
\end{thm}

\begin{proof}
  We put \(\widetilde V(\delta)=\rho_{m+1}(\delta)\otimes\overline{\omega}_{S,\bfh}\) for \(\delta\in\frakS\).
  By~\eqref{eq:odd-global-restriction},
  \[
    \widetilde V_{m,S}(h)=\bigoplus_{\delta\in\frakS}\widetilde V(\delta).
  \]
  We claim that the image of \(\widetilde V(\delta)\) under \(\widetilde I_{m,S}^{\kappa}(h;\,\cdot)\) is nonzero if and only if \(\delta\in\frakS_0\).

  Suppose first that this image is nonzero.
  The Fourier expansion in Lemma~\ref{lem:archimedean-FJ-factor} gives \(T\in\Her_m(F)_{>0}\) for which \(\calJ_T^S\) does not vanish on \(\widetilde V(\delta)\).
  By Proposition~\ref{prop:odd-nonsupercuspidal-structure} and Lemma~\ref{lem:local-dihedral-theta-whittaker},
  this forces
  \[
    \epsilon_{m+1,v}(S\oplus T)=\delta_v
    \qquad(v\in\Sigma_2).
  \]
  Thus \(\delta\in\frakS_0\) by definition.

  Conversely,
  let \(\delta\in\frakS_0\),
  and choose \(T\in\Her_m(F)_{>0}\) as in the definition of \(\frakS_0\).
  Proposition~\ref{prop:odd-nonsupercuspidal-structure},
  Lemmas~\ref{lem:odd-supercuspidal-properties} and~\ref{lem:local-dihedral-theta-whittaker},
  and~\cite[Appendix~B]{Yamana2020HilbertHermitian} show that the restriction of \(\calW_{S\oplus T}\) to \(\rho_{m+1}(\delta)\) is nonzero.
  The argument in the proof of Proposition~\ref{prop:even-nonvanishing},
  applied to \(\rho_{m+1}(\delta)\),
  shows that the image of \(\widetilde V(\delta)\) is nonzero.
  This proves the claim.

  For each \(\delta\in\frakS_0\),
  Theorems~\ref{thm:local-one-block-FJ} and~\ref{thm:local-dihedral-FJ},
  together with the singleton cases,
  identify the finite part of the representation generated by the image as \(\rho_m(\delta^S)\).
  Its archimedean part is \(\calD_{m,\infty}^{\hol}\) by construction,
  and it occurs with multiplicity one by Corollary~\ref{cor:global-packet-identification} and Proposition~\ref{prop:mok-global-multiplicity}.
\end{proof}

\begin{cor}\label{cor:S-independence}
  For \(S,S'\in F_{>0}\),
  \[
    \Pi_{m,S}(h)\simeq\Pi_{m,S'}(h).
  \]
\end{cor}

\begin{proof}
  Put \(a=S'/S\),
  and define
  \[
    \delta^{(a)}=(\delta_v\chi_v(a))_{v\in\Sigma_2}.
  \]
  The map \(\delta\mapsto\delta^{(a)}\) permutes \(\frakS_0\).
  Indeed,
  this is immediate if \(\Sigma_1\neq\varnothing\),
  while for \(\Sigma_1=\varnothing\) it follows from \(\prod_{v\in\Sigma_2}\chi_v(a)=1\),
  which is a consequence of the product formula and the positivity of \(a\).
  Since \(\delta^{S'}=(\delta^{(a)})^S\),
  the result follows from Theorem~\ref{thm:exact-finite-decomposition}.
\end{proof}

We henceforth denote this common isomorphism class by \(\Pi_m(h)\).

\section{Comparison with Ikeda's even-degree Hermitian lift over \texorpdfstring{\(\bbQ\)}{Q}}
\label{sec:Q-comparison}

We specialize the construction to \(F=\bbQ\) and compare it with the even-degree Hermitian lift of~\cite{Ikeda2008Lifting}.
By Corollary~\ref{cor:S-independence},
we take \(S=1\) throughout this section.

Let \(F=\bbQ\),
let \(E=K\) be an imaginary quadratic field of discriminant \(-D\),
write \(\calO_K\) for its ring of integers,
and put \(\chi=\chi_{K/\bbQ}\).
For every finite prime \(p\),
set
\begin{align*}
  K_{K,p}^{(m)}
   & =G_m(\bbQ_p)\cap\GL_{2m}(\calO_{K_p}), \\
  K_{K,\bfh}^{(m)}
   & =\prod_{p<\infty}K_{K,p}^{(m)}.
\end{align*}
Then
\[
  \Gamma_K^{(m)}
  =G_m(\bbQ)\cap K_{K,\bfh}^{(m)}
  =\U_{m,m}(\bbQ)\cap\GL_{2m}(\calO_K).
\]

Let \(m=2n\),
and let \(f\in S_{2k+1}(\Gamma_0(D),\chi)\) be a normalized primitive form.
We specialize the preceding construction by taking \(h=f\),
so that \(\kappa=2k+1\) and \(\omega_f=\chi\),
and choose \(\vartheta^+=\chi_-\).
This is allowed because \(\chi_-|_{\bbA_\bbQ^\times}=\chi=\omega_f\).
Then \(\vartheta^-=\one\).

We first describe the local components of \(\pi_f\) as principal series.
By~\cite[\S~9]{Ikeda2008Lifting},
for every prime \(p\) there is an unramified unitary character \(\mu_p:\bbQ_p^\times\rightarrow\bbC^\times\) such that
\[
  \pi_{f,p}
  \simeq
  \Ind_{B_2(\bbQ_p)}^{\GL_2(\bbQ_p)}
  \left(
  \mu_p\boxtimes\chi_p\mu_p^{-1}
  \right).
\]
Put
\[
  \alpha_p=\mu_p(p).
\]
For \(p\nmid D\),
this is the usual Satake parameter.
For \(p\mid D\),
this agrees with the convention in~\cite{Ikeda2008Lifting}
\[
  \alpha_p=p^{-k}a(p),
\]
where \(a(p)\) is the \(p\)-th Hecke eigenvalue of \(f\);
see~\cite[Introduction]{Ikeda2008Lifting}.

For every prime \(p\),
the definitions give
\begin{equation}
  \label{eq:Q-zeta-specialization}
  \zeta_p^+=\chi_{-,p}\cdot(\mu_p\circ N_{K_p/\bbQ_p}),
  \ \text{and}\
  \zeta_p^-=\mu_p\circ N_{K_p/\bbQ_p}.
\end{equation}
In particular,
\(\zeta_p^-\) is unramified even when \(p\mid D\).
Moreover,
every \(\pi_{f,p}\) is nonsupercuspidal,
and at a nonsplit prime
\begin{equation}
  \label{eq:Q-xi-specialization}
  \xi_p=\left.\zeta_p^+\right|_{\bbQ_p^\times}=\chi_p\mu_p^2.
\end{equation}
If \(p\mid D\),
then \(\xi_p\) is ramified,
while if \(p\nmid D\),
it is unitary.
Since \(\xi_p\neq|\cdot|_p\) and every \(\pi_{f,p}\) is nonsupercuspidal,
\begin{equation}
  \label{eq:Q-local-type-specialization}
  \Sigma_{\st}=\Sigma_{\sc}=\Sigma_{\dh}=\varnothing.
\end{equation}

For an inert prime \(p\nmid D\),
the corresponding assertions follow from the calculation at good places above.
It remains to treat the ramified primes.

\begin{lem}
  \label{lem:Q-full-level-local}
  Let \(p\mid D\).
  Then
  \[
    \dim_\bbC I_{m,p}(\zeta_p^-)^{K_{K,p}^{(m)}}=1.
  \]
  If \(I_{m,p}(\zeta_p^-)\) is reducible,
  that is,
  \[
    I_{m,p}(\zeta_p^-)
    =
    A_{m,p}^{+}(\zeta_p^-)
    \oplus
    A_{m,p}^{-}(\zeta_p^-),
  \]
  then the \(K_{K,p}^{(m)}\)-fixed subspaces of its constituents satisfy
  \[
    \dim_\bbC A_{m,p}^{+}(\zeta_p^-)^{K_{K,p}^{(m)}}=1,
    \ \text{ and }\
    A_{m,p}^{-}(\zeta_p^-)^{K_{K,p}^{(m)}}=0.
  \]
\end{lem}

\begin{proof}
  Since \(\zeta_p^-\) is unramified and unitary,
  the Iwasawa decomposition in~\cite[\S~5]{Yamana2020HilbertHermitian} shows that evaluation at the identity identifies the space of \(K_{K,p}^{(m)}\)-fixed sections with \(\bbC\).
  Suppose that \(I_{m,p}(\zeta_p^-)\) is reducible.
  For a split unimodular \(H_m\in\Her_m(\bbQ_p)\),
  the results in~\cite[Proposition~5.4(2), (4), equation~(9.1), and Lemma~9.2]{Yamana2020HilbertHermitian} show that \(w_{H_m,p}^{\zeta_p^-}\) is nonzero on \(I_{m,p}(\zeta_p^-)^{K_{K,p}^{(m)}}\),
  but vanishes on \(A_{m,p}^-(\zeta_p^-)\).
  The one-dimensional fixed space therefore lies in \(A_{m,p}^+(\zeta_p^-)\),
  proving the remaining assertions.
\end{proof}

At a split prime,
the local representation is spherical and its \(K_{K,p}^{(m)}\)-fixed space is one-dimensional.
At a nonsplit prime,
the calculation at good places for \(p\nmid D\) and Lemma~\ref{lem:Q-full-level-local} for \(p\mid D\) show that the \(K_{K,p}^{(m)}\)-fixed line belongs to the \(+\)-constituent whenever the local principal series is reducible.
Since~\eqref{eq:Q-local-type-specialization} gives \(\Sigma_2=\Sigma_{\pr,2}\),
put
\[
  \delta^\circ=(1)_{p\in\Sigma_2}\in\frakS.
\]
The preceding calculations of local fixed vectors therefore give
\begin{equation}
  \label{eq:Q-full-level-local-selection}
  \dim_\bbC\rho_m(\delta)^{K_{K,\bfh}^{(m)}}
  =
  \begin{cases}
    1 & (\delta=\delta^\circ),    \\
    0 & (\delta\neq\delta^\circ).
  \end{cases}
\end{equation}

\begin{thm}
  \label{thm:Q-full-level}
  Let \(m=2n\).
  Then
  \[
    \dim_\bbC
    \Pi_{2n,1}(f)^{K_{K,\bfh}^{(2n)}}
    =
    \begin{cases}
      0
       & (n\text{ is odd and }f\text{ has CM by }K), \\[2pt]
      1
       & (\text{otherwise}).
    \end{cases}
  \]
\end{thm}

\begin{proof}
  By Theorem~\ref{thm:exact-finite-decomposition} and~\eqref{eq:Q-full-level-local-selection},
  the fixed space is one-dimensional if \(\delta^\circ\in\frakS_0\) and zero otherwise.

  Suppose first that \(f\) does not have CM by \(K\).
  We claim that \(\Sigma_1\neq\varnothing\).
  Otherwise every nonsplit finite place would lie in \(\Sigma_{\pr,2}\),
  so \(\xi_p=\chi_p\) there.
  By~\eqref{eq:Q-xi-specialization},
  we would have \(\mu_p^2=1\).
  Twisting
  \[
    \pi_{f,p}\simeq
    \Ind_{B_2(\bbQ_p)}^{\GL_2(\bbQ_p)}
    \left(\mu_p\boxtimes\chi_p\mu_p^{-1}\right)
  \]
  by \(\chi_p\) would therefore exchange the two characters at every nonsplit \(p\),
  while \(\chi_p\) is trivial at every split \(p\).
  Strong multiplicity one for \(\GL_2\)~\cite{JacquetShalika1981EulerII} would imply \(\pi_f\simeq\pi_f\otimes\chi\),
  contrary to the assumption.
  Hence \(\Sigma_1\neq\varnothing\),
  and Lemma~\ref{lem:global-sign-condition} shows that \(\delta^\circ\in\frakS_0\).

  Now suppose that \(f\) has CM by \(K\).
  Remark~\ref{rem:global-CM-local-types} and~\eqref{eq:Q-local-type-specialization} give \(\Sigma_2=\{p<\infty\mid K_p/\bbQ_p\text{ is a field}\}\) and \(\Sigma_1=\varnothing\).
  Since \(\delta_p^\circ=1\) for every \(p\in\Sigma_2\),
  Lemma~\ref{lem:global-sign-condition} shows that \(\delta^\circ\in\frakS_0\) if and only if
  \[
    1=(-1)^{m(m+1)/2}=(-1)^n.
  \]
  This holds exactly when \(n\) is even.
\end{proof}

\begin{rem}
  Suppose that \(n\) is odd and \(f\) has CM by an imaginary quadratic field \(K'\).
  The classical component indexed by the identity then vanishes by~\cite[Corollary~15.20]{Ikeda2008Lifting}.
  The adelic lift itself vanishes if and only if \(K'=K\) by~\cite[Corollary~15.21]{Ikeda2008Lifting},
  exactly as in the exceptional case of Theorem~\ref{thm:Q-full-level}.
\end{rem}

For \(T\in\Her_{2n}(\bbQ)_{>0}\),
put
\[
  \gamma(T)=(-D)^n\det T,
\]
and,
for each prime \(p\),
let \(F_p(T,X)\) be the Hermitian Siegel-series polynomial of \cite[\S~2]{Ikeda2008Lifting}.
Following the convention of \cite[(11.1)]{Yamana2020HilbertHermitian},
put
\[
  \widetilde F_p(T,X)
  =
  X^{-\ord_p(\gamma(T))}
  F_p(T,p^{-2n}X^2).
\]

For \(w_p=h_p\otimes\Phi_p \in I_{2n+1,p}(\zeta_p^+)\otimes\overline{\omega_{1,p}}\),
put
\[
  \calJ_{T,p}^{1}(w_p)
  =
  \int_{M_{1,2n}(K_p)}
  \overline{\Phi_p(x)}\,
  w_{1\oplus T,p}^{\zeta_p^+}\!\left(
  I_{2n+1,p}(\zeta_p^+)
  (u_{2n+1,1}(x,0,0))h_p
  \right)
  d\mu_{1,p}(x).
\]
We write \(\doteq\) for equality up to a nonzero scalar independent of \(T\).
By \cite[Corollary~6.4]{Yamana2020HilbertHermitian},
together with the split case in \cite[Appendix~B]{Yamana2020HilbertHermitian},
\begin{equation}\label{eq:Q-global-local-J-factorization}
  \calJ_T^1\!\left({\bigotimes_p}'w_p\right)
  \doteq
  \prod_{p<\infty}
  \chi_p(-1)^n
  \mu_p(\det T)^{-1}
  \calJ_{T,p}^{1}(w_p).
\end{equation}

Assume that \(n\) is even or that \(f\) does not have CM by \(K\),
and choose
\[
  0\neq\psi^K
  \in
  \Pi_{2n,1}(f)^{K_{K,\bfh}^{(2n)}}.
\]
Following \cite{Ikeda2008Lifting},
put
\[
  \Lift^{(2n)}(f)
  =
  I_{2n,1}^{2k+1}(f;\psi^K).
\]

\begin{prop}
  The following assertions hold.
  \begin{enumerate}
    \item The adelic form \(\Lift^{(2n)}(f)\) has level \(K_{K,\bfh}^{(2n)}\) and weight
          \[
            (\ell_{2n},\eta_{2n})=(2k+2n,k+n).
          \]

    \item If \(f\) does not have CM by \(K\),
          its global Arthur parameter is
          \[
            \psi_{2n}(f)=\BC_{K/\bbQ}(\pi_f)[2n].
          \]
          If \(\pi_f\simeq\AI_{K/\bbQ}(\theta)\),
          then
          \[
            \psi_{2n}(f)
            =\theta[2n]\boxplus\theta^c[2n].
          \]
          This is the parameter anticipated in~\cite[\S~18]{Ikeda2008Lifting}.

    \item After scaling \(\psi^K\),
          the classical component indexed by the identity is
          \[
            \Lift^{(2n)}(f)_{I_{2n}}(Z)
            =\sum_{T\in\Her_{2n}(\bbQ)_{>0}}
            |\gamma(T)|^k
            \prod_{p<\infty}\widetilde F_p(T,\alpha_p^{-1})
            e_\infty(\tr(TZ)).
          \]
          Hence \(\Lift^{(2n)}(f)\) is the even-degree Hermitian Ikeda lift of~\cite{Ikeda2008Lifting}.

    \item The standard \(L\)-function satisfies
          \[
            L(s,\Lift^{(2n)}(f),\Std)
            =
            \prod_{i=1}^{2n}
            L\left(s+k+n-i+\frac12,f\right)
            L\left(s+k+n-i+\frac12,f\otimes\chi\right).
          \]
  \end{enumerate}
\end{prop}

\begin{proof}
  Since \(\kappa=2k+1\),
  \(m=2n\),
  and \(\vartheta^-=\one\),
  we have \((\ell_{2n},\eta_{2n})=(2k+2n,k+n)\).
  Thus (1) follows from the choice of \(\psi^K\) and the classical--adelic correspondence.
  Assertion (2) follows from Proposition~\ref{prop:global-parameter-all-cases}.

  We prove (3).
  By Theorem~\ref{thm:Q-full-level},
  \(\Pi_{2n,1}(f)^{K_{K,\bfh}^{(2n)}}\) is one-dimensional.
  For each prime \(p\),
  choose
  \[
    w_p^K
    \in
    I_{2n+1,p}(\zeta_p^+)\otimes\overline{\omega_{1,p}}
  \]
  so that \(\beta_{1,p}^{\chi_-}(w_p^K)\) spans the \(K_{K,p}^{(2n)}\)-fixed line in the \(p\)-component of \(\rho_{2n}(\delta^\circ)\).
  Such a choice is possible by \cite[Proposition~5.4(1) and Corollary~10.2]{Yamana2020HilbertHermitian} at split primes and by Theorem~\ref{thm:local-one-block-FJ} at nonsplit primes.
  Taking the standard unramified vector for almost all \(p\),
  put
  \[
    w^K={\bigotimes_p}'w_p^K\in\widetilde V_{2n,1}(f).
  \]
  Its image spans \(\Pi_{2n,1}(f)^{K_{K,\bfh}^{(2n)}}\);
  after scaling \(\psi^K\),
  we may assume that \(\psi^K\) is this image.

  By \cite[Lemmas~9.2 and~10.1]{Yamana2020HilbertHermitian},
  using \(\zeta_p^-(p)=\alpha_p^2\),
  \[
    \calJ_{T,p}^{1}(w_p^K)
    \doteq
    |\det T|_p^{n+1/2}
    F_p(T,p^{-2n}\alpha_p^2).
  \]
  With the standard choices at almost all primes \(p\),
  the omitted local scalar is equal to \(1\).
  Thus the product of the omitted local scalars is a nonzero scalar independent of \(T\).
  Since \(\mu_p\) is unramified and \(\mu_p(p)=\alpha_p\),
  \eqref{eq:Q-global-local-J-factorization} gives
  \[
    \calJ_T^1(\psi^K)
    \doteq
    \prod_{p<\infty}
    \chi_p(-1)^n
    \alpha_p^{-\ord_p(\det T)}
    |\det T|_p^{n+1/2}
    F_p(T,p^{-2n}\alpha_p^2).
  \]

  By the definition of \(\widetilde F_p\) and its functional equation \cite[Lemma~2.2]{Ikeda2008Lifting},
  the \(p\)-th factor on the right is,
  up to a scalar independent of \(T\),
  \[
    \chi_p(\det T)\,
    |\det T|_p^{n+1/2}
    \widetilde F_p(T,\alpha_p^{-1}).
  \]
  Since \(\prod_{p<\infty}\chi_p(\det T)=1\),
  \[
    \calJ_T^1(\psi^K)
    \doteq
    \prod_{p<\infty}
    |\det T|_p^{n+1/2}
    \widetilde F_p(T,\alpha_p^{-1}).
  \]

  Using~\eqref{eq:even-Fourier-expansion-J},
  the product formula,
  and \(\ell_{2n}=2k+2n\),
  we obtain
  \[
    |\det T|_\infty^{(\ell_{2n}+1)/2}\calJ_T^1(\psi^K)
    \doteq
    |\det T|_\infty^k
    \prod_{p<\infty}\widetilde F_p(T,\alpha_p^{-1}).
  \]
  Since
  \[
    |\gamma(T)|^k
    =
    D^{nk}|\det T|_\infty^k,
  \]
  rescaling \(\psi^K\) gives the Fourier expansion in~(3).

  Assertion~(4) follows from~(2) and the factorization
  \[
    L(s,\BC_{K/\bbQ}(\pi_f))
    =
    L(s,\pi_f)\,
    L(s,\pi_f\otimes\chi).
  \]
\end{proof}
\bibliographystyle{amsplain}
\bibliography{even_hermitian_ikeda}

@book{ArthurClozel1989BaseChange,
  author    = {James Arthur and Laurent Clozel},
  title     = {Simple Algebras, Base Change, and the Advanced Theory of the Trace Formula},
  series    = {Annals of Mathematics Studies},
  volume    = {120},
  publisher = {Princeton University Press},
  address   = {Princeton, NJ},
  year      = {1989},
  doi       = {10.1515/9781400882403}
}

@article{AtobeChidaIbukiyamaKatsuradaYamauchi2023Harder,
  author  = {Hiraku Atobe and Masataka Chida and Tomoyoshi Ibukiyama and Hidenori Katsurada and Takuya Yamauchi},
  title   = {{Harder's} conjecture {I}},
  journal = {Journal of the Mathematical Society of Japan},
  volume  = {75},
  number  = {4},
  pages   = {1339--1408},
  year    = {2023},
  doi     = {10.2969/jmsj/87988798}
}

@article{ChenGan2025TwistedGGP,
  author  = {Rui Chen and Wee Teck Gan},
  title   = {Twisted {Gan--Gross--Prasad} conjecture for certain tempered {$L$}-packets},
  journal = {Journal of the Institute of Mathematics of Jussieu},
  volume  = {24},
  number  = {1},
  pages   = {17--39},
  year    = {2025},
  doi     = {10.1017/S1474748024000197}
}

@article{ChenZou2021LLC,
  author  = {Rui Chen and Jialiang Zou},
  title   = {Local Langlands Correspondence for Unitary Groups via Theta Lifts},
  journal = {Representation Theory},
  volume  = {25},
  pages   = {861--896},
  year    = {2021},
  doi     = {10.1090/ert/588}
}

@article{ChenZou2024BigTheta,
  author  = {Rui Chen and Jialiang Zou},
  title   = {Big Theta Equals Small Theta Generically},
  journal = {Acta Mathematica Sinica, English Series},
  volume  = {40},
  number  = {3},
  pages   = {717--730},
  year    = {2024},
  doi     = {10.1007/s10114-024-3236-5}
}

@article{DukeImamoglu1996Converse,
  author  = {William Duke and {\"O}zlem Imamo{\u{g}}lu},
  title   = {A converse theorem and the {Saito--Kurokawa} lift},
  journal = {International Mathematics Research Notices},
  volume  = {1996},
  number  = {7},
  pages   = {347--355},
  year    = {1996},
  doi     = {10.1155/S1073792896000220}
}

@article{Dummigan2017Congruences,
  author  = {Neil Dummigan},
  title   = {Lifting puzzles and congruences of {Ikeda} and {Ikeda--Miyawaki} lifts},
  journal = {Journal of the Mathematical Society of Japan},
  volume  = {69},
  number  = {2},
  pages   = {801--818},
  year    = {2017},
  doi     = {10.2969/jmsj/06920801}
}

@article{GanIchino2016GrossPrasad,
  author  = {Wee Teck Gan and Atsushi Ichino},
  title   = {The {Gross--Prasad} conjecture and local theta correspondence},
  journal = {Inventiones Mathematicae},
  volume  = {206},
  number  = {3},
  pages   = {705--799},
  year    = {2016},
  doi     = {10.1007/s00222-016-0662-8}
}

@misc{Higashitani2026Periods,
  author = {Jin Higashitani},
  title  = {On the Periods of {Ikeda--Yamana} Lift for the Unitary Group {I}},
  year   = {2026},
  note   = {arXiv:2605.18294v2}
}

@article{Ikeda1994Jacobi,
  author  = {Tamotsu Ikeda},
  title   = {On the theory of {Jacobi} forms and {Fourier--Jacobi} coefficients of {Eisenstein} series},
  journal = {Journal of Mathematics of Kyoto University},
  volume  = {34},
  number  = {3},
  pages   = {615--636},
  year    = {1994},
  doi     = {10.1215/kjm/1250518935}
}

@article{Ikeda2001Siegel,
  author  = {Tamotsu Ikeda},
  title   = {On the lifting of elliptic cusp forms to {Siegel} cusp forms of degree {$2n$}},
  journal = {Annals of Mathematics},
  series  = {2},
  volume  = {154},
  number  = {3},
  pages   = {641--681},
  year    = {2001},
  doi     = {10.2307/3062143}
}

@article{Ikeda2008Lifting,
  author  = {Tamotsu Ikeda},
  title   = {On the lifting of {Hermitian} modular forms},
  journal = {Compositio Mathematica},
  volume  = {144},
  number  = {5},
  pages   = {1107--1154},
  year    = {2008},
  doi     = {10.1112/S0010437X08003643}
}

@article{IkedaYamana2020HilbertSiegel,
  author  = {Tamotsu Ikeda and Shunsuke Yamana},
  title   = {On the lifting of {Hilbert} cusp forms to {Hilbert--Siegel} cusp forms},
  journal = {Annales Scientifiques de l'\'{E}cole Normale Sup\'{e}rieure},
  series  = {4},
  volume  = {53},
  number  = {5},
  pages   = {1121--1181},
  year    = {2020},
  doi     = {10.24033/asens.2442}
}

@article{JacquetShalika1981EulerII,
  author  = {Herv\'e Jacquet and Joseph A. Shalika},
  title   = {On {Euler} products and the classification of automorphic forms. {II}},
  journal = {American Journal of Mathematics},
  volume  = {103},
  number  = {4},
  pages   = {777--815},
  year    = {1981},
  doi     = {10.2307/2374050}
}

@article{Karel1979Functional,
  author  = {Martin L. Karel},
  title   = {Functional equations of {Whittaker} functions on {$p$}-adic groups},
  journal = {American Journal of Mathematics},
  volume  = {101},
  number  = {6},
  pages   = {1303--1325},
  year    = {1979},
  doi     = {10.2307/2374142}
}

@article{Katsurada2017HermitianPeriod,
  author  = {Hidenori Katsurada},
  title   = {On the period of the {Ikeda} lift for {$U(m,m)$}},
  journal = {Mathematische Zeitschrift},
  volume  = {286},
  number  = {1--2},
  pages   = {141--178},
  year    = {2017},
  doi     = {10.1007/s00209-016-1758-y}
}

@article{KatsuradaKawamura2015Period,
  author  = {Hidenori Katsurada and Hisa-aki Kawamura},
  title   = {{Ikeda}'s conjecture on the period of the {Duke--Imamo{\u{g}}lu--Ikeda} lift},
  journal = {Proceedings of the London Mathematical Society},
  series  = {3},
  volume  = {111},
  number  = {2},
  pages   = {445--483},
  year    = {2015},
  doi     = {10.1112/plms/pdv011}
}

@article{KudlaSweet1997Degenerate,
  author  = {Stephen S. Kudla and Sweet, Jr., William J.},
  title   = {Degenerate principal series representations for {$U(n,n)$}},
  journal = {Israel Journal of Mathematics},
  volume  = {98},
  pages   = {253--306},
  year    = {1997},
  doi     = {10.1007/BF02937337}
}

@book{Mok2015Endoscopic,
  author    = {Chung Pang Mok},
  title     = {Endoscopic Classification of Representations of Quasi-Split Unitary Groups},
  series    = {Memoirs of the American Mathematical Society},
  volume    = {235},
  number    = {1108},
  publisher = {American Mathematical Society},
  address   = {Providence, RI},
  year      = {2015},
  doi       = {10.1090/memo/1108}
}

@book{Neukirch1999ANT,
  author    = {J\"urgen Neukirch},
  title     = {Algebraic Number Theory},
  series    = {Grundlehren der mathematischen Wissenschaften},
  volume    = {322},
  publisher = {Springer-Verlag},
  address   = {Berlin},
  year      = {1999},
  doi       = {10.1007/978-3-662-03983-0}
}

@article{Shahidi1981Certain,
  author  = {Freydoon Shahidi},
  title   = {On certain {$L$}-functions},
  journal = {American Journal of Mathematics},
  volume  = {103},
  number  = {2},
  pages   = {297--355},
  year    = {1981},
  doi     = {10.2307/2374219}
}

@article{Shahidi1990Plancherel,
  author  = {Freydoon Shahidi},
  title   = {A proof of {Langlands}' conjecture on {Plancherel} measures; complementary series of {$p$}-adic groups},
  journal = {Annals of Mathematics},
  series  = {2},
  volume  = {132},
  number  = {2},
  pages   = {273--330},
  year    = {1990},
  doi     = {10.2307/1971524}
}

@article{Yamana2020HilbertHermitian,
  author  = {Shunsuke Yamana},
  title   = {On the lifting of {Hilbert} cusp forms to {Hilbert-Hermitian} cusp forms},
  journal = {Transactions of the American Mathematical Society},
  volume  = {373},
  number  = {8},
  pages   = {5395--5438},
  year    = {2020},
  doi     = {10.1090/tran/8096}
}
\end{document}